\documentclass{amsart}
\usepackage{amssymb}
\usepackage[english]{babel}

\newcommand{\cP}{{\mathcal{P}}}

\newcommand{\Ps}{{\mathbf{P}}}
\newcommand{\Z}{{\mathbf{Z}}}
\newcommand{\C}{{\mathbf{C}}}

\newcommand{\Q}{{\mathbf{Q}}}

\newcommand{\cQ}{{\mathcal{Q}}}

\renewcommand{\phi}{\varphi}
    \newtheorem{lemma}{Lemma}[section]
    \newtheorem{proposition}[lemma]{Proposition}
    \newtheorem{theorem}[lemma]{Theorem}
    
    \newtheorem{corollary}[lemma]{Corollary}

   \theoremstyle{definition}
    \newtheorem{definition}[lemma]{Definition}
    \newtheorem{notation}[lemma]{Notation}
    \newtheorem{example}[lemma]{Example}
    \newtheorem{remark}[lemma]{Remark}
    
    \DeclareMathOperator{\rank}{rank}

\DeclareMathOperator{\sing}{{sing}}

\DeclareMathOperator{\prim}{{prim}}

\DeclareMathOperator{\red}{{red}}

\DeclareMathOperator{\tor}{{tor}}

\DeclareMathOperator{\MW}{{MW}}
\DeclareMathOperator{\NS}{{NS}}
\DeclareMathOperator{\End}{{End}}

\DeclareMathOperator{\coker}{{coker}}

\DeclareMathOperator{\Gr}{{Gr}}
\DeclareMathOperator{\spa}{{span}}

\usepackage[all]{xy}
\begin{document}
\title[Elliptic threefolds with constant $j$-invariant]{On the classification of degree 1 elliptic threefolds with constant $j$-invariant}
\author{Remke Kloosterman}
\address{Institut f\"ur Al\-ge\-bra\-ische Geo\-me\-trie, Leibniz Universit\"at
Hannover, Wel\-fen\-gar\-ten 1, 30167 Hannover, Germany}
\curraddr{Institut f\"ur Mathematik, Humboldt Universit\"at zu Berlin, Under den Linden 6, 10099 Berlin, Germany}

\email{klooster@math.hu-berlin.de}

\begin{abstract} We describe the possible Mordell-Weil groups for degree 1 elliptic threefold with rational base and constant $j$-invariant.
Moreover,  we classify all such elliptic threefolds if the $j$-invariant is nonzero.
We can use this classification to describe a class of singular hypersurfaces in $\Ps(2,3,1,1,1)$ that admit no variation of Hodge structure. (Remark~\ref{TorRem}.)
\end{abstract}
\date{\today}
\thanks{The author would like to thank Matthias Sch\"utt and Orsola Tommasi for providing comments on a previous version of this paper. The author would like to thank an anonymous referee for pointing out a mistake in a previous version of the proof of Proposition~\ref{prpSpec}}
\maketitle

\section{Introduction}\label{secInt}
In  this paper we work over the field $\C$ of complex numbers.
Let $\pi: X\to B$ be an elliptic threefold with a (fixed) section $\sigma_0:B \to X$, such that $B$ is a rational surface. Assume that $X$ is not birationally equivalent to a product $E\times B$, with $E$ an elliptic curve.

Fix now a Weierstrass equation for the generic fiber of $\pi$. As explained in Section~\ref{secPrelim} this establishes a degree $6k$ hypersurface $Y\subset \Ps(2k,3k,1,1,1)$  that is birational to $X$ and such that the fibration $\pi$ is birationally equivalent to projection from $(1:1:0:0:0)$ onto a plane.

This integer $k$ is not unique. We call the minimal possible $k$ for which such an $Y$ exists the \emph{degree} of $\pi:X\to B$. One can easily show that if $X$ is a rational threefold then the degree equals 1 or 2, and that if $X$ is Calabi-Yau then the degree equals 3.

For a general point $p\in B$ we can calculate the $j$-invariant of the elliptic curve $\pi^{-1}(p)$. This yields a rational function $j(\pi): B\dashrightarrow \Ps^1$.

In this paper we study elliptic threefolds of degree 1 with rational base and constant $j$-invariant.
We would like to classify all such possible threefolds. The two invariants that interest us  are the configuration of singular fibers of $\pi$ and  the structure of the Mordell-Weil group $\MW(\pi)$, the group of rational sections of $\pi$. 
The actual classification we are aiming at  in this paper is a classification of possible singular loci of irreducible and reduced degree $6$ threefolds $Y$ in $\Ps(2,3,1,1,1)$ together with the possibilities for $\MW(\pi)$. In \cite{MirEllThree} it is explained how to obtain an elliptic threefold $X$ from $Y$.

One way of constructing elliptic threefolds is taking a cone $Y$ over an elliptic surface $S\subset \Ps(2,3,1,1) \subset \Ps(2,3,1,1,1)$. The Mordell-Weil group and the configuration of singular fibers can be obtained  from $S$. All possibilities for such $S$ have already been classified by Oguiso and Shioda  \cite{OS}.
We refer to such $Y$ as `obtained by the cone construction'. We exclude such $Y$ from our classification.
One can show that $Y$ is a cone over an elliptic surface if and only if the discriminant curve is a union of lines through one point.

We split our considerations in three cases, namely the general one $j(\pi)\neq0,1728$, and two special cases $j(\pi)=1728$ and $j(\pi)=0$.

The case $j(\pi)\neq 0,1728$ is the easiest one. In this case it is well-known  that $Y$ is given by
\[ y^2=x^3+AP^2x+BP^3\]
with $A,B\in\C$ and $P\in \C[z_0,z_1,z_2]_2$, i.e., $P$ is homogeneous of degree 2. Our assumptions on $Y$ imply that $P=0$ is a smooth conic. It turns out that in this case $\MW(\pi)\cong(\Z/2\Z)^2$.

The exceptional cases $j(\pi)=0,1728$ are more interesting. In these cases one has an equation of the form
\[ y^2=x^3+R, \mbox{ resp. } y^2=x^3+Qx  \]
with $Q\in\C[z_0,z_1,z_2]_{4}$ and $R\in\C[z_0,z_1,z_2]_{6}$. 

To calculate the group $\MW(\pi)$ we use the results of \cite{ell3HK}. It turns out that $\MW(\pi)$ is determined by the type of singularities and the configuration of singular points of $Q=0$, resp., $R=0$.

More precisely, the main result of \cite{ell3HK} states that $\MW(\pi)$ is isomorphic to the group of Weil divisors on $Y$ modulo the Cartier divisors on $Y$. In our case this can be reformulated as
\[ \rank \MW(\pi)=h^4(Y)_{\prim}=\dim \coker \left( F^{2} H^4(\Ps\setminus Y,\C) \to \oplus_{p\in \cP} H^4_p(Y,\C) \right) ,\]
where $\cP$ consists of the points $\{x=y=Q=0\}_{\sing}$, respectively, $\{x=y=R=0\}_{\sing}$.

The Poincar\'e residue map yields a natural surjection from  $\C[z_0,z_1,z_2]_2 x \oplus\C[z_0,z_1,z_2]_4$ onto $ F^{2} H^4(\Ps\setminus Y,\C)$. To determine $H^4_p(Y,\C)$ we use three methods. Let $p\in \cP$.
\begin{enumerate}
\item If $(Y,p)$ is an isolated singularity and is semi-weighted homogeneous, then there is a method of Dimca to compute an explicit basis for $H^4_p(Y,\C)$, together with the Hodge filtration.
\item If $(Y,p)$ is not weighted homogeneous, but  $(Y,p)$ is isolated, then there is a classical method of Brieskorn \cite{BrieMon} to calculate $H^4_p(Y)$. This method does not produce the Hodge filtration, and in the weight homogeneous case it is more complicated than  Dimca's method. 

This method is implemented in the computer algebra package Singular. Since this case is rather exceptional, we preferred to calculate $H^4_p(Y,\C)$ by using Singular. Hence several of the results in the sequel are only valid up to the correct implementation of Brieskorn's method in Singular.
\item If $(Y,p)$ is a non-isolated singularity, but is weighted homogeneous, then the transversal type is an $ADE$-surface singularity. To calculate $H^4_p(Y,\C)$ we apply a generalization of Dimca's method, due to Hulek and the author \cite{ell3HK}.
\end{enumerate}

We list now the possible groups:

\begin{theorem}\label{MWGrpThm}
Suppose $Y \subset \Ps(2,3,1,1,1)$ is a degree $6$ hypersurface, corresponding to an elliptic threefold $\pi:X \to B$, not obtained by the cone construction and not birational to a product $E\times B$. Then $\MW(\pi)$ is one of the following
\begin{itemize}
\item $(\Z/2\Z)^2$ if $j(\pi)\neq 0,1728$.
\item $(\Z/2\Z),(\Z/2\Z)^2$ or $(\Z/2\Z)\times \Z^2$ if $j(\pi)=1728$.
\item $0,\Z/3\Z, (\Z/2\Z)^2, \Z^2, \Z^4,\Z^6$ if $j(\pi)=0$.
\end{itemize}
\end{theorem}

In the case $j(\pi)=1728$ we get a complete classification:
\begin{theorem}
Suppose $Y$ satisfies the conditions of the previous theorem, and suppose that $j(\pi)=1728$.

We have that $\MW(\pi)\cong (\Z/2\Z)^2$ if and only if $Q=0$ defines a double conic and $\MW(\pi)\cong \Z/2\Z\times \Z^2$ if and only if $Q=0$ is the unique quartic with two $A_3$ singularities.
\end{theorem}

For $j(\pi)=0$ the number of cases to consider is  quite large. One should apply the following program:
\begin{enumerate}
\item Determine all possible types of singularities of sextic curves. This is done in \cite{SexticClass}.
\item For each type of singularity, determine $H^4_p(Y)$.
\item Determine which combinations of singularities are possible on a sextic curve. Here one might restrict oneself to combinations of singularities that yield non-trivial $H^4_p(Y)$.
\item For each configuration, study the relation between $h^4(Y)$ and the position of the singularities.
\end{enumerate}

The second point is completely done in this paper, except for six types of singularities that are both not weighted homogeneous and not isolated. The number of cases to consider at the third and fourth point is quite big. We restrict ourselves to the case where the sextic is non-reduced, and the case where the sextic has ordinary cusps. (It turns out that if the sextic has a node at $p$ then $H^4_p(Y)$ vanishes, for this  reason we study sextic with cusps.)

The curves with only cusps as singularities yield examples for the  groups $0,\Z^2,\Z^4$ and $\Z^6$.
One can show that $\MW(\pi)=(\Z/2\Z)^2$ if and only if $R$ defines a triple conic, and $\MW(\pi)=\Z/3\Z$ if and only if $R$ defines a double cubic. This suffices to provide examples for each of the groups mentioned in 
Theorem~\ref{MWGrpThm}.

In Section~\ref{secPrelim} we recall several standard facts concerning elliptic threefolds. In particular we construct our model $Y$. In Section~\ref{secMWpos} we limit the possibilities for the group $\MW(\pi)$. This is done by studying the behaviour of $\MW(\pi)$ under specialization and considering the classification of rational elliptic surfaces \cite{OS}. In Section~\ref{secSing} we discuss the possible singularities for quartic and sextic plane curves. This yields a classification of possible singularities on $Y$. In Section~\ref{secCalc} we calculate the local cohomology $H^4_p(Y)$ for each possible type of singularity on $Y$. In Section~\ref{secMWrank} we give some details on how to calculate $\rank \MW(\pi)$. In the following three sections we give a classification for the cases $j(\pi)=1728$, $j(\pi)=0$ and $R=0$ is non-reduced and $j(\pi)=0$ and $R=0$ is a cuspidal sextic.
In Section~\ref{SecPos} we prove Theorem~\ref{MWGrpThm}.

\begin{notation}
Let $x,y,z_0,z_1,z_2$ be coordinates on $\Ps(2,3,1,1,1)$. 
Throughout this paper $Y\subset \Ps(2,3,1,1,1)$ is a reduced and irreducible degree $6$ hypersurface, containing the point $(1:1:0:0:0)$, and such that $Y$ corresponds to an elliptic fibration with constant $j$-invariant, i.e., $Y$ has a defining equation of the form
\[ y^2=x^3+AP^2x+BP^3, \; y^2=x^3+Qx, \mbox{ or } y^2=x^3+R.\]
Here $P,Q,R$ are homogeneous polynomials in $z_0,z_1,z_2$ of degree $2,4$ and $6$ respectively and $A,B\in \C\setminus \{0\}$.
The curve $C$ is the curve defined by $P=0$, $Q=0$ or $R=0$ (depending on the case).
The set $\cQ$ consists of the singular points of $C_{\red}$.
\end{notation}

\section{Preliminaries} \label{secPrelim}

\begin{definition} An \emph{elliptic $n$-fold} is a quadruple
$(X,B,\pi,\sigma_0)$, with $X$ a smooth projective $n$-fold, $B$ a smooth
projective $n-1$-fold, $\pi:X\to B$ a flat morphism, such that the generic fiber is
a genus 1 curve and $\sigma_0$ is a section of $\pi$.

The \emph{Mordell-Weil group} of $\pi$, denoted by $\MW(\pi)$, is the group of 
rational sections $\sigma: B \dashrightarrow  X$ with identity element
$\sigma_0$.
\end{definition}

We will focus on the cases $n=2,3$. Note that in the case $n=2$ any rational section can be extended to a regular section.

Clearly $\MW(\pi)$ is a birational invariant, in the sense that if $\pi_i:
X_i\to B_i$, $i=1,2$ are elliptic $n$-folds such that there exists a birational
isomorphism $\varphi: X_1\stackrel{\sim}{\dashrightarrow} X_2$ mapping the general
fiber of $\pi_1$ to the general fiber of $\pi_2$, then $\varphi^*: \MW(\pi_2) \to
\MW(\pi_1)$ is well-defined and is an isomorphism.

We shall now describe in some detail how to associate to
an elliptic $n$-fold $\pi: X \to B$ a degree $6k$  hypersurface $Y$ in the weighted projective
space $\Ps:=\Ps(2k,3k,1^{n-1})$ which is birational to
$X$.
Here we restrict
ourselves to the case where $B$ is a rational $n-1$-fold.  In this case,
the morphism $\pi$ establishes $\C(X)$ as a field extension of
$\C(B)=\C(z_1,\dots,z_{n-1})$.
The field $\C(X)$ is the function field of an elliptic curve $E$ over
$\C(z_1,\dots,z_{n-1})$,
i.e., $\C(X)=\C(x,y,z_1,\dots,z_{n-1})$ where \begin{equation}\label{eqnWei}
y^2=x^3+f_1(z_1,\dots,z_{n-1})x+f_2(z_1,\dots,z_{n-1})\end{equation} with $f_1,f_2\in
\C(z_1,\dots,z_{n-1})$.
One has a natural isomorphism
\[ \MW(\pi)\cong E(\C(B))\]
where $E(\C(B))$ is the group of $\C(B)$-rational points of $E$.

Without loss of generality we may assume that (\ref{eqnWei}) is a global minimal
Weierstrass equation, i.e.,  $f_1,f_2$ are polynomials and there is no
polynomial $g\in \C[z_1,\dots,z_{n-1}]\setminus \C$ such that $g^4$ divides $f_1$ and $g^6$ divides
$f_2$. 
 
To obtain a hypersurface in $\Ps$ we need to find a weighted
homogeneous polynomial. Let $k=\lceil \max  \{\deg(f_1)/4,\deg(f_2)/6\} \rceil$ 
and define $P$ and $Q$ as the polynomials 
 \[P=z_0^{4k} f_1(z_1/z_0,\dots,z_{n-1}/z_0), \quad Q=z_0^{6k} f_2(z_1/z_0,\dots,z_{n-1}/z_0).\]
  Then
 \[ y^2=x^3+P(z_0,z_1,\dots,z_{n-1})x+Q(z_0,z_1,\dots,z_{n-1})\]
 defines a hypersurface $Y$ of degree $6k$ in $\Ps$. Let
$\Sigma$ be the locus 
where all the partial derivatives of the defining equation vanish.
Consider the projection $\tilde{\psi}: \Ps \dashrightarrow \Ps^{n-1}$
with center
$L=\{z_0=\dots=z_{n-1}=0\}$ and its restriction $\psi=\tilde{\psi}|_Y$ to $Y$. 
Then there exists a diagram
\[
 \xymatrix{X \ar@{-->}[r] \ar[d]^\pi & Y\ar@{-->}[d]^\psi \\ S \ar@{-->}[r]&
\Ps^ 2.  }
 \] 

Note that $Y\cap L=\{(1:1:0:\dots:0)\}$. 
If $k=1$ then $\Ps_{\sing}$ consists of two points, none of which lie on
$Y$. If $k>1$ then  an easy calculation in local coordinates shows that
$\Ps_{\sing}$
is precisely $L$, that $\Sigma$ and $L$ are disjoint  and that $Y$ has an
isolated singularity at 
$(1:1:0:\dots:0)$. 
For any $k$ we have that $\psi$ is not defined at $(1:1:0:\dots:0)$. 
Let $\tilde{\Ps}$ be the blow-up of $\Ps$ along $L$.
Let $X_0$ be the strict transform of $Y$ in $\tilde{\Ps}$. An easy calculation
in local coordinates shows that $X_0\to Y$ resolves the singularity of $Y$ at
$(1:1:0:\dots:0)$ and that the induced map $\pi_0: X_0 \to \Ps^2$ is a morphism. 
Moreover, all fibers of $\pi_0$ are irreducible
curves.

\begin{definition}
The degree of an elliptic $n$-fold $\pi:X\to B$, with rational base, is the smallest $k$ such that there is a degree $6k$ hypersurface $Y$ in $\Ps(2k,3k,1^{n-1})$ birational to $\pi$.
\end{definition}

As remarked above, we can consider the generic fiber of $\pi$ as an elliptic curve $E$ over $\C(z_1,\dots,z_{n-1})$. In the sequel we consider only elliptic curves such that $j(\pi)=j(E)$ is constant, i.e., $j(E)\in \C$. 
Most of the sequel will be concentrated on $j(\pi)\in \{0,1728\}$. If this is the case then $E$ has complex multiplication.

\begin{lemma}\label{evenrank} Let $K$ be a field, $E/K$ an elliptic curve, such that $E$ has complex multiplication over $K$. Suppose $\rank E(K)$ is finite. Then $\rank_{\Z} E(K)$ is even.
\end{lemma}
\begin{proof}
Since $E(K)$ is a $\End(E)$-module it follows that
\[ E(K) \cong E(K)_{\tor} \oplus \End(E)^r.\]
Since $\End(E)$ is a free $\Z$-module of rank 2, it follows that $E(K)\cong E(K)_{\tor}\oplus \Z^{2r}$ (as $\Z$-modules), hence $E(K)$ has even $\Z$-rank.
\end{proof}

The following minor results will be used several times:

\begin{lemma}\label{lem4tor} Let $V/\C$ be a variety. Let $E/\C(V)$ be an elliptic curve such that $j(E)\in \C$. Suppose $j(E)\neq 0,1728$, then $(\Z/2\Z)^2$ is a subgroup of $E(\C(V))$.
\end{lemma}

\begin{proof} Let $E'/\C$ be an elliptic curve with $j(E')=j(E)$. Then we can find a Weierstrass equation $y^2=x^3+ax+b$ for $E'$, with $a,b\in \C$. Let $\alpha_1,\alpha_2,\alpha_3$ be roots of $x^3+ax+b$.

Since $j(E)\neq0,1728$ we have that $E$ is given by 
\[y^2=x^3+aP^2x+bP^3\]
for some $P\in \C(V)^*$. For $i=1,2,3$ we have that $x=\alpha_iP$ is a root of $x^3+aP^2x+bP^3$, hence $x=\alpha_iP,y=0$ is a point of order 2 on $E(\C(V))$. From this it follows that $(\Z/2\Z)^2\subset E(\C(V))$.
\end{proof}

\begin{lemma}\label{lem2tor}
 Let $K$ be any field not of characteristic $2,3$. Let $E/K$ be an elliptic curve with $j(E)=1728$ then $E(K)$ contains a point of order 2.
\end{lemma}
\begin{proof}
Since $K$ is not of characteristic $2,3$ we have that $E$ has a Weierstrass equation $y^2=x^3+ax$ with $a\in K$. The point $(0,0)$ is a point of order 2.
\end{proof}

\section{Possible Mordell-Weil groups \& Specialization} \label{secMWpos}
We describe now all possible Mordell-Weil groups for elliptic surfaces of degree 1 with  constant $j$-invariant. Using a specialization result this limits the possibilities for Mordell-Weil groups for elliptic threefolds of degree 1.
Note that an elliptic surface is rational if and only if its degree is 1. 
We start by recalling some results from Oguiso and Shioda \cite{OS}:

\begin{proposition} \label{prpSur0} Let $\pi :S\to \Ps^1$ be a rational elliptic surface with $j(\pi)=0$. Then
$S$ is birationally equivalent to a surface in $\Ps(2,3,1,1)$ given by an equation of the form $y^2=x^2+f(z_0,z_1)$, with $f$ homogeneous of degree $6$. In the following table we list all possible factorizations of $f$, the contribution of the singular fibers to the N\'eron-Severi lattice, $\rank \MW(\pi)$ and $\# \MW(\pi)_{\tor}$.
\[
\begin{array}{|c|c|c|c|}
\hline
\mbox{Factorization}&  \mbox{Cont. to } \NS(S) & \rank \MW(\pi) &\# \MW(\pi)_{\tor}\\
\hline
{[}1,1,1,1,1,1{]}    & -    & 8 &1\\
{[}2,1,1,1,1]  & A_2  & 6 &1\\
{[}2,2,1,1] & 2A_2 & 4& 1\\
{[}2,2,2]    & 3A_2 &2 &3\\
{[}3,1,1,1]  &  D_4  & 4 &1\\
{[}3,2,1]   & D_4+A_2& 2 &1\\
{[}3,3]     & 2D_4   & 0 &4\\
{[}4,1,1]  & E_6    & 2 &1\\
{[}4,2]   & E_6+A_2 & 0 & 3\\
{[}5,1]    & E_8  & 0 & 1\\
\hline
\end{array}
\]
\end{proposition}

\begin{proposition}\label{prpSur1728}
Let $\pi :S\to \Ps^1$ be a rational elliptic surface with $j(\pi)=1728$. Then
$S$ is birationally equivalent to a surface in $\Ps(2,3,1,1)$ given by an equation of the form $y^2=x^2+f(z_0,z_1)x$, with $f$ homogeneous of degree $4$. In the following table we list all possible factorizations of $f$, the contribution of the singular fibers to the N\'eron-Severi lattice, $\rank \MW(\pi)$ and $\# \MW(\pi)_{\tor}$.
\[
\begin{array}{|c|c|c|c|}
\hline
\mbox{Factorization}&  \mbox{Cont. to} \NS(S) & \rank \MW(\pi) &\# \MW(\pi)_{\tor}\\
\hline
{[} 1,1,1,1]   & 4A_1    & 4 & 2\\
{[} 2,1,1] & D_4+2A_1 & 2 &2\\
{[} 2,2]   & 2D_4    & 0 & 4\\
{[} 3,1]   & E_7+A_1 & 0 & 2 \\
\hline
\end{array}
\]
\end{proposition}
\begin{proposition}\label{propSurGen}
Let $\pi :S\to \Ps^1$ be a rational elliptic surface with $j(\pi)$ constant and $j(\pi)\neq 0,1728$. Then $S$ is birationally equivalent to a surface in $\Ps(2,3,1,1)$
given by an equation of the form $y^2=x^2+af(z_0,z_1)^2x+bf(z_0,z_1)^3$, with $\deg(f)=2$ and $a,b\in \C$. Then $\pi$ has two fibers with corresponding lattice $D_4$ and $\MW(\pi)\cong (\Z/2\Z)^2$.
\end{proposition}

Therefore, for a rational elliptic surface with constant $j$-invariant the  possible Mordell-Weil groups are $\Z^{r}$, $r\in \{0,2,4,6,8\}$; $(\Z/2\Z) \times \Z^{r_2}$, $r_2\in\{0,2,4\}$; $(\Z/3\Z)\times \Z^{r_3}$, $r_3\in \{0,2\}$ and $\Z/4\Z$.

We will prove now a specialization result, probably well-known to the experts, that implies that in the threefold case, the Mordell-Weil group is a subgroup of one of the above groups.

Let $Y\subset \Ps(2,3,1,1,1)$ be an elliptic threefold. Let $\ell=\{a_0z_0+a_1z_1+a_2z_2=0\} \subset \Ps^2$ be a line. Let $H_\ell=\{a_0z_0+a_1z_1+a_2z_2=0\} \subset \Ps$ be the corresponding hyperplane. Then $Y_\ell=Y \cap H_\ell\subset \Ps(2,3,1,1)$ is a rational elliptic surface, provided $\ell$ was not a component of the discriminant curve of $\pi$.

\begin{lemma} The restriction of rational sections to $\ell$ defines a group homomorphism $\MW(\pi) \to \MW(\pi_\ell)$.

\end{lemma}
\begin{proof}
A rational section $\sigma: \Ps^2 \dashrightarrow Y$ can be represented as \[[z_0,z_1,z_2] \mapsto [ f,g, z_0h,z_1h,z_2h]\] where $f$, $g$, $h$ are homogeneous polynomials in $z_0$,$z_1$ and $z_2$ such that $\deg(f)=2\deg(h)+2$ and $\deg(g)=3\deg(h)+3$. 

The indeterminacy locus of such a rational map is contained in $V(f,g,h)$, the locus where $f,g,h$  vanish. We prove now that we can find representatives $f,g,h$ such  that $\gcd(f,g,h)=1$, i.e., the indeterminacy locus of $\sigma$ is finite. This implies that $\sigma|_\ell$ is a well-defined rational section.

Assume we have an irreducible non-constant polynomial $p$ dividing $f,g,h$. Then $p^3$ divides
\[ f^3+P fh^4+Q h^6=g^2,\]
hence $p^4$ divides $g^2$. Write $g=p^2g_1$, $h=ph_1$ and $f=pf_1$. Now,
\[ p^4g_1^2= g^2= f^3+Pfh^4+Q h^6=p^3(f_1^3+Af_1h_1^4p^2+Bh_1^6 p^3). \]
From this it follows that $p$ divides $f_1$ and $p^2$ divides $f$. This implies that $p^6$ divides $f^3+P fh^4+Q h^6=g^2$, hence $p^3$ divides $g$ and
\[[z_0,z_1,z_2] \mapsto [ f/p^2,g/p^3, z_0h/p,z_1h/p,z_2h/p]\]
is a well-defined rational map, represented by polynomials and equivalent to the rational map
$[z_0,z_1,z_2] \mapsto [ f,g, z_0h,z_1h,z_2h]$.

Iterating this process shows that we may assume that $V(f,g,h)$ is a finite set. Since the indeterminacy locus of $\sigma$ is finite, the restriction  $\sigma|_\ell: \ell\to Y_\ell$ is a rational map, which can be extended to a morphism, since $\ell$ is a curve. If $h$ does not vanish at all points of $\ell$ then $\sigma_\ell$ is a section. If $h$ does vanish along $\ell$, then the image of $\sigma_\ell$ lies in $Y_\ell\cap \{z_0=z_1=z_2=0\} = \{(1:1:0:0:0\}$, i.e, the image of $\sigma|_\ell$ coincides with the image of the zero-section. This yields the existence of the map $\MW(\pi)\to \MW(\pi_\ell)$.

Finally, to see that this map is a group homomorphism, note that on both $Y$ and $Y_\ell$ the addition of sections is defined fiber-wise.
\end{proof}

The following result is probably known to the experts, but we did not find a reference for this in the literature.
\begin{proposition}\label{prpSpec}
Let $\Delta_{\red}\subset \Ps^2$ be the reduced curve defined by the vanishing of $4A^3+27B^2$. 
Let $\ell\subset \Ps^2$ be a very general line. Then the map $\MW(\pi)\to \MW(\pi_\ell)$ is injective.

Moreover, suppose that $\Delta_{\red}$ is neither a union of lines nor an irreducible conic. Then there exists a line $\ell$ such that 
\begin{enumerate}
\item $\ell$ is tangent to $\Delta_{\red}$ at some point.
\item $\ell$ intersects $\Delta_{\red}$ in at least two distinct points.
\item The natural map
\[\MW(\pi)\to \MW(\pi_\ell)\]
is injective.
\end{enumerate}
\end{proposition}
\begin{proof}
It suffices to prove that there are at most countable many lines such that
\[ \MW(\pi)\to \MW(\pi_\ell)\]
is not injective, since if $\Delta_{\red}$ is not the union of lines nor a conic then there are uncountable many lines that satisfy the first and second property the results follows.

Let $r=\rank \MW(\pi)$. Write $\MW(\pi)_{\tor}=\{\tau_1,\dots,\tau_k\}$. Fix $\sigma_1,\dots,\sigma_r\in \MW(\pi)$ such that the $\sigma_i$ and $\tau_j$ generate $\MW(\pi)$.

Consider a section $\sigma=\tau_j+\sum_{i=1}^r n_i \sigma_i\in \MW(\pi)$. Write $\sigma$ as 
\[[z_0,z_1,z_2] \mapsto [ f,g, z_0h,z_1h,z_2h]\] where $f$, $g$, $h$ are homogeneous polynomials in $z_0$,$z_1$ and $z_2$ such that $\deg(f)=2\deg(h)+2$ and $\deg(g)=3\deg(h)+3$. 

If $\sigma$ lies in the kernel of $\MW(\pi)\to \MW(\pi_\ell)$ then $h$ needs to vanish along $\ell$. In particular, there are at most finitely many lines $\ell$ such that $\sigma$ is mapped to zero in $\MW(\pi_\ell)$. Since $\MW(\pi)$ is countable there are at most countably many lines $\ell$ for which some $\sigma$ is mapped to zero, i.e., for which the map $\MW(\pi)\to \MW(\pi_\ell)$ is not injective.
\end{proof}

\begin{corollary}\label{corMWjGen}
Suppose $\pi: X \to B$ is an elliptic threefold of degree 1, $j(\pi)$ is constant and $j(\pi)\neq 0,1728$. Then $\MW(\pi)=(\Z/2\Z)^2$.
\end{corollary}

\begin{proof} From Lemma~\ref{lem4tor} it follows that $(\Z/2\Z)^2\subset \MW(\pi)$.
From Propositions~\ref{propSurGen} and~\ref{prpSpec} it follows that $\MW(\pi)\subset (\Z/2\Z)^2$, which yields the corollary.
\end{proof}

\begin{corollary}\label{corMWj1728}
Suppose $\pi: X \to B$ is an elliptic threefold of degree 1, $j(\pi)$ is constant and equals $1728$. Then $\MW(\pi)$ is one of the following:
\[(\Z/2\Z)\times \Z^{r}, (\Z/2\Z)^2 \]
with $ r\in\{0,2,4\}$.
\end{corollary}

\begin{proof} From Lemma~\ref{lem2tor} it follows that $(\Z/2\Z)\subset \MW(\pi)$. From Lemma~\ref{evenrank} it follows that the rank is even.
From Propositions~\ref{prpSur1728} and~\ref{prpSpec} it follows that $\MW(\pi)$ is a subgroup of either $(\Z/2\Z)\times \Z^4$ or $(\Z/2\Z)^2$, which yields the corollary.
\end{proof}

\begin{corollary}\label{corMWj0}
Suppose $\pi: X \to B$ is an elliptic threefold of degree 1 and $j(\pi)$ is constant and equals $0$. Then $\MW(\pi)$ is one of the following:
\[\Z^{r_1} , (\Z/3\Z)   \times \Z^{r_2}, (\Z/2\Z)^2 \]
with $r_1\in \{0,2,4,8\}$, $r_2\in \{0,2\}$ .
\end{corollary}

\begin{proof} 
From Propositions~\ref{prpSur1728} and~\ref{prpSpec} it follows that $\MW(\pi)$ is a subgroup of either $\Z^8$,   $(\Z/2\Z)^2$ or $\Z/3\Z \times \Z^2$.
From Lemma~\ref{evenrank} it follows that the rank is even. 

To prove the corollary we have to exclude the group $\Z/2\Z$. Suppose $\pi$ has a section of order two, i.e., $Y$ is given by an equation of the from $y^2=x^3-T^3$ and $[z_0,z_1,z_2]\mapsto [T,0,z_0,z_1,z_2]$ is a section of order 2. Then  for $i=1,2$ the morphisms $[z_0,z_1,z_2]\mapsto [\omega^i T,0,z_0,z_1,z_2]$ define also sections of order 2, where $\omega^2=-\omega-1$, hence we have complete two-torsion. In particular, $ \Z/2\Z$ does not occur as possible Mordell-Weil group.
\end{proof}

Later on we will show that the cases $r=4$, $r_1=8$ and $r_2=2$ can only occur in the cone construction case. At this point we will show that in these cases the curve $C$ is a union of lines, but not necessarily through one point.

Assume that the $j$-invariant is constant. Then it is well-known that $\rank \MW(\pi_\ell)$ equals $2a-4$, where $a$ is the number of singular fibers of $\pi_\ell$, i.e., $a=\# C \cap \ell$ (counted without multiplicities).

\begin{lemma}\label{lemj1728lines} Suppose $C$ is not the union of lines. Then $y^2=x^3+Px$ has Mordell-Weil rank at most 2.
\end{lemma}
\begin{proof}
If $P$ is of the form $P_0^2$ for some irreducible polynomial $P_0$ of degree 2. Then for a  very general line $\ell$ the elliptic surface $\pi_\ell$ is an elliptic surface with $2I_0^*$ fibers, and therefore has rank 0. Since for a very  general line the map $\MW(\pi)\to \MW(\pi_{\ell})$ is injective (Proposition~\ref{prpSpec}) we have that $\rank \MW(\pi)=0$.

Otherwise we can apply the second part of Proposition~\ref{prpSpec}. Let $\ell$ be a line satisfying the properties mentioned in this proposition. Then $\pi_{\ell}$ has $j=1728$ and one fiber either of type $I_0^*$ or $III^*$ and hence by Proposition~\ref{prpSur1728} has rank at most 2. Since $\MW(\pi)\to \MW(\pi_\ell)$ is injective, it follows that $\MW(\pi)$ has rank at most 2.   \end{proof}

\begin{lemma}\label{lemj0lines} Suppose $C$ is not the union of lines. Then $y^2=x^3+Q$ has Mordell-Weil rank at most 6.
\end{lemma}
\begin{proof}
The proof is similar to the previous lemma. If $Q$ is a quadratic polynomial to the power three then for a general line $\pi_\ell$ is an elliptic surface with 2 $I_0^*$ fibers has therefore rank 0 and hence $\rank \MW(\pi)=0$.

Otherwise we can apply Proposition~\ref{prpSpec}. Using this we find a line $\ell$ such that $\pi_\ell$ has a fiber of type $IV, I_0^*, IV^*$ or $II^*$, and therefore $\rank \MW(\pi)_\ell \leq 6$, and such that $\MW(\pi) \to \MW(\pi_{\ell})$ is injective.\end{proof}

\begin{lemma}\label{lem3tors}
Suppose $j=0$ then $3\mid \#\MW(\pi)_{\tor}$ if and only if $Y$ is given by an equation of the form $y^2=x^3+f^2$, where $f$ is a cubic polynomial. 

If $(\Z/3\Z)\times \Z^2$ is a subgroup of $\MW(\pi)$ then $f=0$ is a union of  lines.
\end{lemma}
\begin{proof}
Suppose $3\mid \# \MW(\pi)_{\tor}$. Since for a very  general line $\ell$ the map $\MW(\pi)\to \MW(\pi_\ell)$ is injective (Proposition~\ref{prpSpec}) it follows from Propositions~\ref{prpSur0}  that for such a  line $\ell$ the intersection $C\cap\ell$ consists of three points with multiplicity 2 or one point with multiplicity 2 and one point with multiplicity 4. Hence $C$ is a double cubic. 

Conversely, if  $Y$ is given by $y^2=x^3+f^2$  then  $x=0,y=f$ defines a section of order 3.

Suppose $f=0$ is not the union of lines. Then $C$ is a reduced cubic. Then by Proposition~\ref{prpSpec} there exists a line $\ell$, such that $\MW(\pi)\to \MW(\pi_\ell)$ is injective and such that $\# \ell \cap C=2$, hence $\MW(\pi_\ell)$ is finite.
\end{proof}

\section{Singularities of quartic and sextic plane curves} \label{secSing}

\subsection{Quartic curves}
The classification of singular quartic curves is well-known. We give a sketch. First assume that $C$ is reduced. Then either $C$ is the union of four lines through one point, or $C$ has at most $ADE$ singularities.  The first case corresponds to the cone construction case, so suppose we are in the latter case. Let $p_1,\dots,p_k$ be the singularities of $C$, let $M_1,\dots,M_k$ be the corresponding Milnor lattices. Then $\oplus M_i$ can be embedded in the lattice corresponding to the affine Dynkin diagram $\tilde{E}_7$. This limits the possibilities to $A_1,\dots,A_7, D_4\dots,D_7$ and $E_6,E_7$.

Assume that $C$ is non-reduced and that $C_{\red}$ is not the union of lines through one point. Then $C$ is either a double line $\ell$ together with a (possible reducible) conic $T$, or a double conic. If $C$ is a double conic, then it has to be irreducible, hence $C_{\red}$ is smooth and $\cP=\emptyset$.
 
Let $q$ be a point of the singular locus of $C_{\red}$. The above discussion shows that $(C,q)$ is one of the following singularities
\begin{itemize}
\item $t^2s$, i.e., $\ell$ and $T$ intersect transversely;
\item $t^2(t-s^2)$,  i.e., $\ell$ is tangent to $T$;
\item $ts$,  ($A_1$ singularity), i.e., $T$ is the union of two lines.
\end{itemize}
Note that in  the second case we have that $\cP$ consists of one point.

 \subsection{Sextic curves}
 Sextic curves have more possible singularities. 
 \begin{theorem}\label{ThmSextic}
 A reduced sextic can have the following singularities \cite{SexticClass}:
\begin{itemize} 
\item $A_k: x^2+y^{k+1}$, $k\leq 19$.
\item $D_k: y(x^2+y^{k-1})$, $k\leq 19$.
\item $E_6: x^3+y^4$.
\item $E_7: x^3+xy^3$.
\item $E_8: x^3+y^5$. 
\item $B_{k,l}: x^k+y^l $, $3\leq k\leq l$. If $k=3$ then  $6\leq l\leq 12$, if $k>3$ then $l\leq 6$.
\item $xB_{k,l}: x(x^k+y^l) $,  $(k,l)\in \{(2,5),(2,7),(3,4),(3,5),(4,5)\}$.
\item $yB_{k,l}: y(x^k+y^l) $,   $(k,l) \in \{(3,4),(3,5),(3,6),(4,5)\}$.
\item $xyB_{k,l}: xy(x^k+y^l) $,  $(k,l) \in \{(2,3),(3,4)\}$.
\item $C_{k,l}: x^k+y^l+x^2y^2$, $k\leq l$, $2/k+2/l\leq 1$ and $k\leq n(l)$, with $n(l)=15,14,14,12,11,11,9$ for $l=3,4,\dots,9$.
\item $yC_{k,l}: y(x^k+y^l+x^2y^2)$. $k\leq l$, $2/k+2/l\leq 1$, $k\in \{3,5\}$. If $k=3$ then $7\leq l\leq 12$.  If $k=5$ then $l \in\{5,6\}$.
\item $D_{k,l}: x^k+y^l+x^2y^3$. $2/p+3/q\leq 1$. If $k=3$ then $9\leq 10\leq 13$. Otherwise $(k,l)\in \{(4,7),(5,6),(5,7),(6,5),(6,6),(6,7)\}$.
\item $F_{k,l}: x^k+y^l+x^2y^3+x^3y^2$. $6\leq k\leq l\leq 7$.
\item $S_{2k-1}:  (x^2+y^3)^2+(a_0+a_1y)xy^{4+k}$. $a_0\neq 0, a_1\in \C, k=1,2,3$.
\item $S_{2k}:  (x^2+y^3)^2+(a_0+a_1y)x^2y^{3+k}$. $a_0\neq 0, a_1\in \C, k=1,2,3$.
\end{itemize}
\end{theorem}

All these singularities are also in Arnol'd's list, so one might also use the names given by Arnol'd. A translation between our name-giving and that of Arnol'd can be found in \cite[Remark 1]{SexticClass}. 

Several of these singularities have distinct Milnor and Tjurina number, and are therefore not semi-weighted homogeneous.

We do not present a classification of non-reduced sextics here. Essentially, one has either 
\begin{itemize}
\item a double line with a quartic,
\item a double conic with another conic,
\item a double cubic,
\item a triple line with a cubic,
\item a triple line with a double line and a line,
\item a triple conic or
\item a quadruple line with conic.
\end{itemize}

The possibilities for the singularities are a combination of the possibilities of singularities for conics, cubics and quartic, and the possible intersection numbers between the components.

\section{Calculating $H^4_p(Y,\C)$}\label{secCalc}

In this section we discuss three approaches to calculate $H^4_p(Y,\Q)$.  For each singularity that we encounter, one of these methods applies, except for six types of singularities. We list $h^4_p(Y)$ for each of the singularities.

\subsection{Dimca's method}
Let $(Y,p)$ be a semi-weighted isolated hypersurface singularity. We have a local equation of the form $f_p+g_p=0$, such that $(Y,p)$ is a $\mu$-constant deformation of $f_p=0$ and $f_p=0$ defines a weighted homogeneous isolated hypersurface singularity. 

Let  $w_1,\dots,w_4$ are the weights of the variables, $w_p=w_1+w_2+w_3+w_4$ and let $d_p$ be the (weighted) degree of $f$. Then Dimca \cite{Dim} shows
\[ H^4_p(Y)=\oplus_{i=1}^3 R(f)_{i d_p-w_p}.\]
Moreover, this direct sum decomposition is just $\oplus \Gr^F_{4-i} H^4_p(Y)$. Finally, Dimca shows that $H^4_p(Y)$ has a pure weight 4 Hodge structure.

It turns out that all singularities under consideration statisfy $R(f)_{d_p-w_p}=0$ and $R(f)_{3d_p-w_p}=0$. This follows from the fact that in all cases $d_p< w_p$ and the existence of a non-degenerated pairing $R(f)_{d_p-w_p}\times R(f)_{3d_p-w_p} \to R(f)_{4d_p-2w_p}\cong \C$.
This implies that $H^4_p(Y,\C)=\C(-2)^k$ with $k=h^4_p(Y)$.

If $j(\pi)=1728$ then all singularities of $Y$ are non-isolated, so for the rest of this subsection assume that $j(\pi)=0$.

We have a semi-weighted hypersurface singularity if and only if the sextic $C$ is reduced at $q=\psi(p)$ and has a weighted homogeneous singularity\footnote{From the discussion of singular sextics it follows that only the $S_k$ singularities have moduli. Since the $S_k$ singularities  are not semi weighted homogeneous it turns out that all semi weighted homogeneous singularities are rigid and therefore weighted homogeneous.}.
This limits us to cases that $C$ has either an $ADE$ singularity, or a
 $B_{k,l},xB_{k,l},yB_{k,l},xyB_{k,l}$ singularity. We list now the  singularities with non-trivial $H^4_p(Y)$.
 
\begin{proposition}
Suppose $(C,q)$  is a weighted homogeneous singularity of a sextic curve, not a point of order six, and such that $h^4_p(Y)\neq 0$. Then $(C,q)$ is one of
\begin{itemize}
\item $A_2,A_5,A_8,A_{11},A_{14},A_{17}$,
\item $E_6$,
\item $B_{3,6},B_{3,8},B_{3,10},B_{3,12},B_{4,6}$.
\end{itemize}
\end{proposition}

The following lemmata yield a proof of this proposition and the proofs provide basis for $H^4_p(Y,\C)$ for each non-trivial case.

\begin{lemma}\label{lemj0Ak}
Suppose $C$ has an $A_k$ singularity at $q$. If $k \equiv 2\bmod 3$ then $H^4_p(Y,\C$ is isomorphic to $\C(-2)^2$. Otherwise,  $H^4_p(Y,\C)$ vanishes.
\end{lemma}

\begin{proof} We have a local equation for $Y$ of the form
\[ y^2=x^3+t^2+s^{k+1}.\]
Setting weights $6,3k+3,2k+2,3k+3$ for $s,t,x,y$, we obtain $d_p=6k+6$, $w_p=8k+14$. Hence $2d_p-w_p=4k-2$. The Jacobian ideal is generated by $y,t,x^2, s^k$. Hence $R(f)_{4k-2}$ is spanned by
\[ xs^{(k-2)/3}, s^{(2k-1)/3} \]
This means that $H^4_p(Y)=0$ if $k\not \equiv 2 \bmod 3$ and $H^4_p(Y)=\C(-2)$ if $k\equiv 2 \bmod 3$.
\end{proof}

\begin{lemma}
Suppose $C$ has an $D_k$ singularity at $q$ then $H^4_p(Y,\C)=0$.
\end{lemma}

\begin{proof} We have a local equation for $Y$ of the form
\[ y^2=x^3+st^2+s^{k-1}.\]
Setting weights $6,3k-6,2k-2,3k-3$, we obtain $d_p=6k-6$, $w_p=8k-5$. Hence $2d_p-w_p=4k-7$. The Jacobian ideal is generated by $y,st,x^2, (k-1)s^{k-2}+t^2$. Since $t$ and $x$ have even weight, it follows that each generator of $R(f)_{4k-7}$ is divisible by $t$. Since $st$ is in the Jacobian ideal, the only possibility is $x^\alpha t^\beta$. Since $(2k-2)\alpha + (3k-6) \beta = 4k-7$ 
has no integral solution with $k>3, \beta>0,\alpha\geq0$ we have $R(f)_{4k-7}=0$.
\end{proof}

\begin{lemma}
Suppose $C$ has an $E_k$ singularity at $q$, $k\in \{6,7,8\}$ then $H^4_p(Y,\C)=\C(-2)^2$ if $k=6$ and $H^4_p(Y,\C)=0$ otherwise.
\end{lemma}

\begin{proof}$E_6:$ We have a local equation for $Y$ of the form
\[ y^2=x^3+t^3+s^4.\]
Setting weights $3,4,4,6$, we obtain $d_p=12, w_p=17, 2d_p-w_p=7$. The only monomials of weights 7 are $xs,ts$ and their classes provide a basis for $R(f_p)_{7}$.

$E_7:$ We have a local equation for $Y$ of the form
\[ y^2=x^3+t^3+s^3t.\]
Setting weights $4,6,6,9$, we obtain $2d_p-w_p=11$. Since there are no monomials of degree 11, we obtain $R(f_p)_{11}=0$.

$E_8:$ We have a local equation for $Y$ of the form
\[ y^2=x^3+t^3+s^5.\]
Setting weights $6,10,10,15$, we obtain $2d_p-w_p=19$. Since there are no monomials of degree 19, we obtain $R(f_p)_{10}=0$.
\end{proof}

\begin{remark}\label{rem3wts}
Suppose $(Y,p)$ is a weighted homogeneous hypersurface singularity. 
Let $(\{f_p=0\},0)$ be a local equation of $(Y,p)$, where $f_p$ is weighted homogeneous. Then $S_p=\{f_p=0\}$ defines a surface in $\Ps(w_0,w_1,w_2,w_3)$. Dimca's method (as well as the method of Hulek-Kloosterman) relies on the isomorphism $H^4_p(Y,\C)\cong H^2(S_p,\C)_{\prim}(1)$.

Often one can simplify this calculation. If exactly three of the four weights have a non-trivial common divisor one can apply the following procedure:
Suppose $S_p\subset \Ps(w_0,gw_1,gw_2,gw_3)$ and $g\nmid w_0$. Then there is an isomorphism  $\varphi: \Ps(w_0,gw_1,gw_2,gw_3)\to \Ps(w_0,w_1,w_2,w_3)$ by sending $(x_0:x_1:x_2:x_3) \to (x_0^g:x_1:x_2:x_3)$. 

Let $g_p$ be the equation of $\varphi(S_p)$. Suppose that $g_p$ has an isolated singularity, i.e, $\varphi(S_p)$ is quasi-smooth. Since $\varphi$ is an isomorphism we have then that $H^{1,1}(S_p)_{\prim}\cong R(g_p)_{d-w_0-w_1-w_2-w_3}$. 
We often refer to this as the \emph{three weights trick}.

The reason to apply the three weights trick is the following: in several cases it turns out that $1$ lies in the Jacobian ideal of $g_p$. This in turn implies that $R(g_p)=0$, and $h^{1,1}(S_p)_{\prim}=0$.
\end{remark}

\begin{lemma}
Suppose $C$ has a $B_{k,l}$ singularity and $6\nmid kl$ then $H^4_p(Y,\C)=0$.
\end{lemma}

\begin{proof}
We have a local equation for $(C,q)$ of the form
\[y^2=x^3+t^k+s^l \]
hence we can set the weights for  $s,t,x,y$ to be $6k,6l,2kl,3kl$. If $2\nmid kl$ we may apply the three weight trick (Remark~\ref{rem3wts}). Therefore it suffices to study
$y=x^3+t^k+s^l$. If $3\nmid kl$ we can also use the three weights trick, in this case we obtain the singularity $y^2=x+t^k+s^l$. In both cases 1 is in the Jacobian ideal, hence the Jacobian ring (and therefore the local cohomology) vanishes.
\end{proof}

Recall that in Theorem~\ref{ThmSextic} we give a list of possible values $(k,l)$ such that $B_{k,l}$ occurs as a singularity on a sextic.
\begin{lemma}
Suppose $C$ has a $B_{k,l}$ singularity and $6\mid kl$ then 
\begin{itemize}
\item $H^4_p(Y,\C)=\C(-2)^2$ if $k=3$ and $l\equiv 2,4\bmod 6$.
\item $H^4_p(Y,\C)=\C(-2)^4$ if $k=3$ and $l\equiv 0\bmod 6$.
\item $H^4_p(Y,\C)=\C(-2)^2$ if $(k,l)=(4,6)$.
\item $H^4_p(Y,\C)=0$ if $(k,l)=(5,6)$.
\end{itemize}
\end{lemma}
\begin{proof}
An easy computation shows that if $(k,l)=(5,6)$ then $R_{2d_p-w_p}=0$ and there is no local cohomology. 

If $k=3$ then $R_{2d_p-w_p}$ is generated by \[xts^{(l-6)/6},xs^{(l-2)/2}, ts^{(l-2)/2},s^{(5l-1)/6}.\]

If $(k,l)=(4,6)$ then $R_{2d_p-w_p}$ is generated by $xts,t^2s$. 
\end{proof}

\begin{lemma}
 Suppose $C$ has an $xB_{k,l}$, an $yB_{k,l}$ or  an $xyB_{k,l}$ singularity then  $H^4_p(Y,\C)$ vanishes.
\end{lemma}
\begin{proof}
We used the computer algebra package Singular to check for every admissible value of $(k,l)$ (see Theorem~\ref{ThmSextic}) that $R_{2d_p-w_p}=0$.
\end{proof}

\subsection{Method of Brieskorn}
A second of class of singularities are non-weighted homogeneous isolated hypersurface singularities.

Let $f_p$ be a local equation for $(Y,p)$. Then $f_p=y^2+x^3+g_p(s,t)$. We explain now the method to calculate $H^4_p(Y)$. First observe that this group equals $H^4(F)^0$ the part of the cohomology of the Milnor fiber that is invariant under the monodromy. 

Now $H^4(F)$ is naturally  isomorphic to the Milnor algebra of $(f_p,0)$. The Milnor algebra can be easily calculated. Brieskorn  \cite{BrieMon} developed a method to calculate the action  of the monodromy on $H^4(F)$. We will not explain this method, but use the computer algebra package Singular, which contains an implementation of this method. 

For computational reasons it is better to let Singular calculate the monodromy action on $H^2(F_1)$,  where $F_1$ is the Milnor fiber of $g_p(s,t)=0$. From this one can deduce the monodromy on $H^4(F)$ as follows:

For an arbitrary singularity $f(x_1,\dots,x_n)=0$ one can identify $H^n(F)$ with the Milnor algebra $M(f):=\C[[x_1,\dots,x_n]]/(f,f_{x_0},\dots,f_{x_0})$. 

Now, the Milnor algebra of $x_{n+1}^d+f$ is the direct sum $\oplus_{i=0}^{d-2} x^iM(f)$.
One easily shows that  for $h\in M(f)$ we have $T_{x_{n+1}^d+f}(x^j h)=\exp(2j \pi i/d) x^j T_{f}(h)$, where $T_{g}$ is the monodromy operator for the singularity $(g,0)$. More specific, to find all the eigenvalues of $T_{x_{n+1}^d+f}$ one needs to multiply  all the eigenvalues of $T_{f}$  by all the $d$-th roots of unity except for the root of unity 1. In the case $y^2+x^3+g_p(s,t)$ we apply this procedure twice. Hence the eigenvalues of the monodromy of $f$ get multiplied by $\exp(5/3\pi i)$ and $\exp(1/3\pi i)$, i.e, the two primitive sixth roots of unity.
So, in order to determine $H^4_p(Y)$ we need to find the eigenspaces for the eigenvalues $\exp(5/3\pi i)$ and $\exp(1/3\pi i)$ on $H^2(F_1)$.

The computer algebra package Singular produced the following results:

\begin{proposition} Suppose $(C,q)$ is a $C_{k,l},yC_{k,l},D_{k,l},F_{k,l} $ or $S_k$ singularity on a sextic curve.
Then 
\begin{itemize}
\item $h^4_p(Y,\C)=4$ if $(C,q)$ is $C_{3k,3l}$ singularity.
\item $h^4_p(Y,\C)=2$ if $(C,q)$ is $C_{k,l}$ singularity, where exactly one of $k,l$ is divisible by $3$, or $(C,q)$ is either a $S_3$ or a $S_6$ singularity.
\item  $h^4_p(Y,\C)=0$ otherwise.
\end{itemize}

\end{proposition}

\subsection{Method of Hulek-Kloosterman}
The final method we use works for non-isolated singularities. Let $(Y,p)$ be such a singularity. Since we have a minimal elliptic threefold, such a singularity is one-dimensional, and the transversal types are $ADE$ surface singularities.

There are three classes to distinguish:
\begin{itemize}
\item $j(\pi)=1728$ and $C$ has an isolated singularity at $q$.
\item $j(\pi)=1728$ and $C$ has a non-isolated singularity at $q$.
\item $j(\pi)=0$ and  $C$ has a non-isolated singularity at $q$.
\end{itemize}

If $C$ is a quartic then $(C,q)$ is weighted homogeneous. If $C$ is a sextic, then except for six types of singularities, $(C,q)$ is weighted homogeneous. For the rest of this subsection assume that $(C,q)$ is weighted homogeneous. Then $(Y,q)$ is weighted homogeneous.
This implies that we may apply \cite[Proposition 7.7]{ell3HK}, which is a generalization of Dimca's method. We start by giving a short outline of this method:

Let $f_p$ be a local equation for $(Y,q)$, let $w_p$ and $d_p$ be as in Dimca's method. Let $R(f_p)$ be the Jacobian ring of $f_p$. Hulek and the author proved that $H^4_p(Y)$ has pure Hodge structure of weight 4 with $h^{4,0}=h^{0,4}=0$, 
$h^{3,1}=h^{1,3}=\dim R(f_p)_{d_p-w_p}$. To determine $h^{2,2}$ we need to introduce some notation.
The equation $f_p=0$ defines a surface $S_p$ in weighted projective 3-space $\Ps'$. 
Now, Hulek and the author show that  $h^{2,2}(H^4_p(Y))=h^{1,1}(S_p)_{\prim}$.

The Hodge number $h^{1,1}(S_p)_{\prim}$ can be determined as follows:
Let $q_1$, $\dots$, $q_s$ be the points where all the partials of $f_p$ vanish. Then $(S_p,q_i)$ is an $ADE$ singularity. If $q\not  \in \Ps'_{\sing}$ then let $M_{q}$ be the Milnor algebra of $(S_p,q)$. 

If $q\in  \Ps'_{\sing}$ we do the following: we have a degree 6 quotient map $\varphi:\Ps^4\to \Ps'$ let $G$ be the Galois group of this cover. Let $\tilde{q}\in \varphi^{-1}(q)$. Let $G_q$ be the stabilizer of $\tilde{q}$. Let $g$ be a local equation of $(S_p,q)$ in $\Ps'$. Then $G_q$ acts on the Milnor algebra of $g$. Let $M_q$ be the invariant part of $M$ under $G_q$. One can show that this definition is independent of the choices made.
Let \[\tilde{R}(f_p)_{2d_p-w_p}:=\ker R(f_p)_{2d_p-w_p} \to \oplus M_{q_i}.\]
Then it follows from the work of Steenbrink \cite{SteAdj} that 
\[ h^{1,1}(S_p)=\dim \tilde{R}(f_p)_{2d_p-w_p}.\]
This suffices to calculate all the Hodge numbers.

\begin{remark}\label{remMiln} In addition to the \emph{three weights trick} (Remark~\ref{rem3wts}) there is another trick we can apply. Namely, let $\Sigma(f_p)$ be the locus, where all the partials of $f_p$ vanish. Assume that $\Sigma(f_p)\cap \Ps'_{\sing}=\emptyset$.
Then 
\[ h^{1,1}(S_p)=h^{1,1}(S)-\sum_{q\in \Sigma(f_p)} \mu_q \]
where $\mu_q$ is the Milnor number of the singularity at $q$. (This formula holds, since $S_p$ has only $ADE$ surface singularities. For a proof of this, see e.g. \cite[Lemma 6.1]{ell3HK}.)
\end{remark}

As written above we distinguish between three classes. First assume $j(\pi)=1728$.

\begin{proposition}
Suppose $(C,q)$  is a singularity of a quartic curve, not a point of order four, such that $h^4_p(Y)\neq 0$. Then $(C,q)$ is isolated and one of $A_3,A_7,D_7$.
\end{proposition}

Assume first that $(C,q)$ is isolated. From the classification of singular quartics it follows that $(C,q)$ is an $ADE$ singularity.

In all cases it turns out that $d_p-w_p<0$, hence $H^4_p(Y)$ is of pure $(2,2)$-type. Since $w_p$ and $d_p$ are listed in every proof, we do not mention that $d_p<w_p$. We prove now:

\begin{lemma}
Suppose $C$ has an $A_k$ singularity at $q$ then $H^4_p(Y,\C)=\C(-2)^2$ if $k \equiv 3\bmod 4$ and $H^4_p(Y,\C)=0$ otherwise.
\end{lemma}

\begin{proof}
If $C$ has an $A_k$ singularity at $q$ then $Y$ is locally of the from $f_p=0$ with 
\[f_p= y^2+x^3+(t^{k+1}+s^2)x. \]
Set the weights of $s,t,x,y$ to be $2k+2,4,2k+2,3k+3$. The sum $w_p$ of the weights equals $7k+11$.
The degree $d_p$ equals $6k+6$.

To determine $h^{2,2}$ we start by determining $R_{2d_p-w_p}=R_{5k+1}$. Since $y,x^2+ t^{k+1}+s^2, t^kx$ and $sx$ generate the Jacobian ideal, it follows that
\[ R_{5k+1}=\spa \{ t^{(5k+1)/4}, t^{(3k-1)/4}s,t^{(k-3)/4}s^2, xt^{(3k-1)/4}\}. \]
Hence $R_{5k+1}=0$ if $k\not \equiv 3 \bmod 4$. From this it follows that $h^4_p(Y)=0$  if $k\not \equiv 3 \bmod 4$.

Suppose  that $k\equiv 3 \bmod 4$, i.e., $k=3+4(m-1)$. Our defining equation is of the form
\[ y^2+x^3+(t^{4m}+s^2)x. \]
We set the weights of $s,t,x,y$ to be $2m,1,2m,3m$. Now, $d_p=6m, w_p=7m+1$. From this it follows that  $R(f_p)_{2d_p-w_p}$ is generated by
\[ t^{5m-1}, t^{3m-1}s, t^{m-1}s^2, t^{3m-1}x.\]
Since $S$ has  $A_1$ singularities at $(1:1:0:0)$ and $(1:-1:0:0)$. The Milnor algebra is generated by 1, i.e.,   $\tilde{R}_{2d_p-w_p}$ is spanned by elements of $R_{2d_p-w_p}$ that vanish at $(1:1:0:0)$ and $(1:-1:0:0)$, hence it is spanned by
\[ xt^{3m-1}\mbox{ and } (t^{4m}-s^2)t^{m-1}\]
and $h^4_p(Y)=2$.
\end{proof}

 \begin{lemma}
Suppose $C$ has a $D_k$ singularity at $q$. Then $H^4_p(Y,\C)=\C(-2)^2$ if $k \equiv 3\bmod 4$ and $H^4_p(Y,\C)=0$ otherwise.
\end{lemma}

 \begin{proof}
In this case we have a local equation of the form
\[ y^2+x^3+t(t^{k-2}+s^2) x. \]
The weights here are $2k-4,4, 2k-2, 3k-3$. Hence $d_p=6k-6$, $w_p=7k-5$, $ 2d_p-w_p=5k-7$.  Consider
\[ R_{2d_p-w_p}=(\C[x,y,t,s]/(y,x^2+t^{k-1}+s^2t, st x, ((k-1)t^{k-2}+s^2)x))_{5k-7}. \] 
It is easy to see that $t^{5/4k-7/4}, t^{3/4k-3/4}s,t^{1/4k+1/4}s^2, xt^{3/4k-5/4}$ span this vector space. Hence, a necessary condition to have local cohomology is $k\equiv \pm 1 \bmod 4$.

Consider first the case $k\equiv 3 \bmod 4$, i.e, $k=4m+3$, then we have a local equation of the form
\[ y^2+x^3+t(t^{4m+1}+s^2) x \]
We can normalize the weights such that they become $4m+1,2,4m+2,6m+3$. The degree is $12m+6$, the sum of the weights equals $14m+8$.
The vector space $R_{10m+4}$ is spanned by
\[ t^{5m+2}, s^2 t^{m+1}, xt^{3m+1}.\]

The partials of $f_p$ vanish if $t=x=y=0$ or  if $t^{4m+1}+s^2=x=y=0$. These equations yield points $q_1,q_2$ where $S_p$ has an $A_1$ singularity. At such a point the Milnor algebra is generated by 1, 
hence the kernel $R(f_p)\to M_{q_1}\oplus M_{q_2}$  consists of functions vanishing at $q_1$ and $q_2$. So $\tilde{R}$ is generated by
\[ (t^{4m+1}+s^2)t^{m+1}, xt^{3m+1}.\]
 Thus $H^4_p(Y,\C)=\C(-2)^2$.

The case $k\equiv 1 \bmod 4$ is different. Set $k=4m+1$. Then we have a local equation of the form
\[ y^2+x^3+t(t^{4m-1}+s^2) x .\]
The weights are $4m-1,2,4m,6m$. This surface is isomorphic to the surface $S$ given by
\[ y^2+x^3+t(t^{4m-1}+s)x \]
in $\Ps(4m-1,1,2m,3m)$. 
The surface $S$ is of degree $6m$ and the sum of the weights is $9m$. The only monomials of degree $3m$ are
$y,xt^m,t^{3m}$. Since $y$ and $xt$ are  in the Jacobian ideal it turns out that $R(f_p)_{2d_p-w_p}$ is generated by $t^{3m}$.

The surface $S$ has an $A_1$ singularity at $q=(1:-1:0:0)$. At this point we have a trivial stabilizer. The Milnor algebra $M_{q}$ is generated by $1$ in local coordinates. Hence all elements of $\tilde{R}(f_p)_{2d_p-w_p}$ have to vanish at $q$. So $t^{3m} \not \in\tilde{R}(f_p)_{2d_p-w_p}$, hence $h^4_p(Y)=0$.
 \end{proof}

 \begin{lemma}
Suppose $C$ has an $E_k$ singularity at $q$  then $H^4_p(Y,\C)$ vanishes.
\end{lemma}

\begin{proof}
Case $E_6$:
\[y^2+x^3 +(s^3+t^4)x,\]
the weights are $4,3,6,9$. This surface is isomorphic to
\[y^2+x^3 +(s+t^4)x\]
in $\Ps(4,1,2,3)$. The degree is 6, the sum of the weights equals $10$, whence $2d_p-w_p=2$. The only monomials of degree 2 are $x$ and $t^2$. Since $x$ is in the Jacobian ideal it follows that $R_{2d_p-w_p}$ is generated by $t^2$. As $S$ has an $A_1$ singularity  at $(1:1:0:0)$, all elements of $\tilde{R}_{2d_p-w_p}$ have to vanish at $(1:1:0:0)$. Since $t^2$ does not vanish, we obtain that $h^4_p(Y)=0$.

Case $E_7$:
\[y^2+x^3+(s^3+st^3)x,\]
the weights are  $12,8,18,27$. This surface is isomorphic to
\[ y+x^3+(s^3+st^3)x\]
in $\Ps(6,4,9,27)$. Since $1$ is in the Jacobian ideal we obtain $R$ is the zero ring, hence $H^4_p(Y,\C)=0$.
\end{proof}

We come now to the case where $(C,q)$ is a not an isolated singularity. From Section~\ref{secSing} it follows that we only have to consider the following two singularities: 

\begin{lemma}
Suppose $(C,q)$ has local equation $s^2t=0$ or $s^2(s-t^2)=0$. Then $H^4_p(Y)=0$.
\end{lemma}

\begin{proof}
In the first case we have a  local equation  $y^2=x^3+s^2tx$ for $(Y,p)$.
This defines a degree 6 surface $S$ in $\Ps(1,2,2,3)$. Hence $2d_p-w_p=4$. The monomials $xt,s^4,t^2,s^2t$ span $R(f_p)_{4}$. The surface $S$ has two singularities, namely at $q_1:=(1:0:0:0)$ and $q_2:=(0:1:0:0)$.

The Milnor algebra $M_{q_1}$ is generated by $1$ (which translates to $s^4$ in global coordinates).
For $q_2$, note that the Milnor algebra is generated by $1,x,s,s^2$. The group $G_{q_2}$ is generated by $s\mapsto -s$, hence $M_{q_2}^{G_{q_2}}$ is spanned by  $t^2,xt,s^2t$ and $\tilde{R}_{2d_p-w_p}=0$.

Consider now $y^2=x^3+s^2(s-t^2)x$. 
This defines a surface in $\Ps(4,2,6,9)$ and is isomorphic to $y=x^3+s^2(s-t^2)x$ in $\Ps(2,1,3,9)$. This surface has $h^2_{\prim}=0$, hence $H^4_p(Y,\C)=0$.
\end{proof}

We turn now to the final case, namely $j(\pi)=0$ and $C$ is non-reduced.

\begin{lemma}\label{lemNull} Suppose $(Y,p)$ is one of the following singularities
\begin{enumerate}
\item $y^2=x^3+t^2s$
\item $y^2=x^3+t^2(t-s^3)$
\item $y^2=x^3+t^3s$
\item $y^2=x^3+t^3(t-s^2)$
\item $y^2=x^3+t^4s$
\item $y^2=x^3+t^2s^2$
\end{enumerate}
Then $H^4_p(Y)=0$
\end{lemma}

\begin{proof}
For each case we list a choice for the weights. We then either state that we may apply the the three weights trick (Remark~\ref{rem3wts}) or we give an outline on how to compute $\tilde{R}_{2d_p-w_p}$:
\begin{enumerate}
\item $2,2,2,3$: (three weights).
\item $2,6,6,9$: (three weights).
\item $3,1,2,3$: In this case we have $2d_p-w_p=3$. A basis  for $R(f_p)_3$ is $s,xt$. At $(1:0:0:0)$ we have the following stabilizer: $x\mapsto \omega^2x, t\mapsto \omega t$. The Milnor algebra has basis  $1,x,t,xt$. After taking invariants under the stabilizer we find that  $1,xt$ span $M_p^{G_p}$. Hence  $\tilde{R}_{2d_p-w_p}=0$.
\item $3,6,8,12$: (three weights).
\item $2,1,2,3$: In this case we have $2d-w=4$. A basis for $R(f_p)_4$ is $xs$, $xt^2$, $t^2s$, $s^2$. At $(1:0:0:0)$ we have  $t\to -t$ as stabilizer. The Milnor algebra is spanned by  $1,x,t,xt,t^2,xt^2$, hence the invariants under the stabilizer are (in global coordinates)   $s^2,xs,t^2s,xt^2$. Hence $\tilde{R}(f_p)_{2d_p-w_p}=0$.
\item $2,1,2,3$: A basis  for $R(f_p)_4$ is $xt^2,xs,t^4,s^2$. 
At $(1:0:0:0)$ the stabilizer is generated by $t\to -t$, the Milnor algebra is spanned  by $1,x$, hence is invariant under the stabilizer, so we can exclude $s^2,xs$.
At $(0:1:0:0)$  we have no stabilizer, the Milnor algebra is spanned by $1,x$ in local coordinates, hence $t^4,xt^2$ can be excluded.
From this it follows that $\tilde{R}_{2d_p-w_p}(f_p)=0$.
\end{enumerate}
\end{proof}

\begin{lemma} \label{lemDblTpl}
Suppose $T$ is a reduced quartic with a double and a triple point. Then either
\begin{itemize}
\item $T$ has exactly two singularities, the triple point is a $D_6$ singularity and the double point is an $A_1$ singularity,
 \item $T$ has exactly two singularities, the triple point is  a $D_5$ singularity and the double point is an $A_1$ singularity or
\item $T$ has a $D_4$ singularity an up to 3 $A_1$ singularities.
\end{itemize}
\end{lemma}

\begin{proof}
Suppose $q_1$ is a double point and $q_2$ a the triple point. Let $\ell$ be the line through $q_1$ and $q_2$. Since $(T.\ell)_{q_1} \geq 2$ and $(T.\ell)_{q_2} \geq 3$ it follows that $\ell$ is a component of $T$. Let $K$ be the residual cubic. 
Then $q_1$ is a smooth point of $K$ and $q_2$ is double point of $K$. Since $T$ is reduced, we have that $\ell$ is not a component of $K$. From this it follows that $(K.\ell)_{q_i}=i$ for $i=1,2$. In particular, at $q_1$ we have an $A_1$ singularity. Hence all double points of $T$ are $A_1$ singularities.

Note that if $K$ has an $A_k$ singularity at $q_1$ then $T$ has $D_{3+k}$ singularity. Since $K$ is a cubic we have that $k\leq 3$.

If $K$ has an $A_3$ singularity at $q_1$ then $K$ is a conic $Q$ together with a line tangent to $Q$ at $q_1$. Hence $T$ has a $D_6$ and an $A_1$ singularity and no other singularities.

If $K$ has an $A_2$ singularity at $q_1$ then $K$ is an irreducible cubic and smooth outside $q_1$. Hence $T$ has a $D_5$ and an $A_1$ singularities.

If $K$ has an $A_1$ singularity at $q_1$ then $K$ has at most 2 other $A_1$ singularities, hence $T$ has a $D_4$ singularity and at most three $A_1$ singularities.
\end{proof}

\begin{lemma} Suppose $C$ is a double line $\ell$ together with a reduced quartic $T$. Suppose that $Y$ has at least two singularities such that $H^4_p(Y)\neq 0$ then one of the following occurs

\begin{itemize}
\item $C$ has at least two cusps, none of them along $\ell$.
\item $\ell$ is a bitangent of $T$, and $C$ might be  smooth or has double points along $C\cap \ell$.
\item $C$ has an $E_6$ singularity, but not along $C\cap \ell$, and there is a point $p\in C\cap \ell$ such that $(C\cdot \ell)_q\in \{2,4\}$.
\item $C$ has an $A_2$ or $A_5$ singularity, $C$ is smooth along $C\cap \ell$ and there is a point $q\in C\cap \ell$ such that $(C\cdot \ell)_q\in \{2,4\}$
\item $C$ has an $A_2$ or $A_5$ singularity not along $C\cap \ell$ and $C$ has a double point along $C\cap \ell$.
\end{itemize}
\end{lemma}

\begin{proof}
Suppose first that $T$ is smooth outside $T\cap \ell$. Since we have at least two singularities that are not rationally smooth, and the singularity $y^2=x^3+t^2s$ is rationally smooth (Lemma~\ref{lemNull}), it follows that $(T\cdot \ell)\geq 2$ for at least 2 points in the intersection. Hence $\ell$ is a bitangent.

Suppose that there is a singularity $(T,q'), q'\not \in \ell$ such that $H^4_p(Y)\neq 0$. Then $(T,q')$ is a $A_2, A_5$ or $E_6$ singularity. Let $q\in T\cap \ell$. If $T$ is smooth at $q$ it follows from Lemma~\ref{lemNull} that $(T\cdot \ell)_q\in \{2,4\}$.

If $(T,q')$ is an $E_6$ singularity then it follows from Lemma~\ref{lemDblTpl} that $T$ has no double points. Since a reduced quartic has at most one triple point, this implies that $T$ is smooth outside $q'$. Hence the second singularity such that $H^4_p(Y)\neq 0$, comes from a point in $q\in T\cap \ell$. From Lemma~\ref{lemNull} it follows that   $(T\cdot \ell)_{q} \in \{2,4\}$.

If $(T,q')$ is an $A_2$ or $A_5$ singularity then $(T,q)$ might be smooth and the intersection number $(T\cdot \ell)_q$ is 2 or 4, or $(T,q)$ is an $A_k$ singularity.

Suppose none of the intersections points of $T$ and $\ell$ yields a non-trivial $H^4_p(Y)$. Then $T$ has at least two singularities with types  $A_2,A_5,E_6$. Since the combinations $2E_6$, $2A_5$, $E_6+A_2$ and $A_5+A_2$ are not possible, it follows that $T$ has at least two $A_2$ singularities.
\end{proof}

\begin{lemma} Suppose $Q$ is a quartic with an $A_k$ singularity at $q$ and $\ell$ is a line through $p$, not contained in $Q$. Then $(k,(Q\cdot \ell))$ is one of the following
\begin{itemize}
 \item $(k,2)$, $1 \leq k\leq 7$.
\item $(k,3)$, $1 \leq k \leq 2$.
\item $(k,4)$, $1 \leq k \leq 7$, $k\neq 2$.
\end{itemize}

\end{lemma}
\begin{proof} Since $Q$ is a quartic, we have $1\leq k \leq 7$. For a general line $\ell$  we have $(Q\cdot \ell)=2$. This yields the case $(k,2)$.

Suppose now $k=2$ and $\ell$ is given by $t=0$. The quartics  locally given by $st+s^3$ or $st+s^4$ yield the cases $(1,3)$ and $(1,4)$. 

Suppose now that we have $k>1$ then we have a local equation of the form
\[ t^2+a_{30}s^3+a_{21}s^2t+a_{12}st^2+a_{03}t^3+a_{40}s^4+a_{31}s^3t+a_{22}s^2t^2+a_{03}st^3+a_{04}t^4.\]
Since $\ell$ is not a component of $Q$, we have that either $a_{30}$ or $a_{40}$ is nonzero.

If $a_{30}\neq 0$ then we have an $A_2$ singularity and $C\cdot \ell=3$.

If $a_{30}=0$ then $k\geq 3$ and $C\cdot \ell=4$. 
\end{proof}

\begin{remark} A straightforward calculation shows that the contact equivalence class of $t^2f(t,s)$, where $f$ is a singularity on a quartic curve, is determined by the type of singularity of $f$ and the intersection number of $t=0$ with $f(t,s)$.
\end{remark}

\begin{lemma} Suppose $C$ is a triple line $\ell$ together with a reduced cubic $K$.
Suppose that $Y$ has at least two singularities such that $H^4_p(Y)\neq 0$ then $K$ is a cuspidal cubic, and $\ell$ is a flex line at a smooth point of $K$.
\end{lemma}

\begin{proof}
Suppose $K$ has two points $q_1,q_2$ not on $\ell$ yielding non-zero $H^4_p(Y)$, then $K$ has an $A_2$ singularity at $q_1,q_2$. Since a cubic has at most one cusp, this is not possible.

Suppose there is a point $q\in K \cap \ell$ yielding a non-trivial $H^4_p(Y)$. Then from Lemma~\ref{lemNull} it follows that $K$ is singular at $q$, or $q$ is a flex point and $\ell$ is a flex line. This implies that there is at least one point $q'$ not on $\ell$ yielding non trivial local cohomology.

Since a cubic has only $A_1,A_2$ or $D_4$ singularities, and $A_1,D_4$ singularities yield rationally smooth points on $Y$, it follows that $(K,q)$ is an $A_2$ singularity.

Since cuspidal cubics have exactly one singularity,  it follows that $(K,q)$ is smooth, hence $\ell$ the flex line of $K$ at $q$.
\end{proof}


\begin{lemma} Suppose $C$ is a quadruple line $\ell$ together with a reduced conic $T$. Then  $Y$ has at most one  singular point $p$ with $H^4_p(Y)\neq 0$.
\end{lemma}

\begin{proof}Let $q_1$ and $q_2$ be points yielding non-trivial local cohomology. Since $T$ is a conic it is either smooth or has an $A_1$ singularity. Since an isolated $A_1$ singularity yields a rational smooth singularity on $Y$, we have that $q_1,q_2\in \ell$.
In particular $\ell$ is not a tangent of $T$. From Lemma~\ref{lemNull} it follows that $H^4_p(Y)=0$ in this case.
\end{proof}

\begin{lemma} Suppose $C$ consists of two double lines $\ell_1,\ell_2$  together with a reduced conic $T$. Suppose that $Y$ has at least two singularities such that $H^4_p(Y)\neq 0$. Then $\ell_1$ and $\ell_2$ are tangent to $T$.
\end{lemma}

\begin{proof} A point  on  $T$ but not in $T\cap (\ell_1\cup \ell_2)$ is either an isolated $A_1$ singularity of $C$ or smooth, hence has no non-trivial local cohomology.

From Lemma~\ref{lemNull} it follows that transversal intersections of  $\ell_1$ with $T$ has trivial local cohomology. Hence $\ell_1$ and $\ell_2$ are tangent to $C$.
\end{proof}

\begin{lemma} Suppose $C$ consists of a smooth  double conic $K$  together  with a reduced conic $T$. Suppose that $Y$ has at least two singularities such that $H^4_p(Y)\neq 0$ then $C$ and $K$ have common tangents at two intersections points.
\end{lemma}

\begin{proof} A point  on  $T$ but not in $T\cap K$ is either an isolated $A_1$ singularity of $C$ or smooth, hence has trivial local cohomology.

Transversal intersections of $K$ with $T$ have trivial local cohomology. Hence we need at least two points such that $(K\cdot T)_q\geq 2$. Since $K$ and $T$ are conics, this implies that $K$ and $T$ have two intersections points with intersection multiplicity 2.
\end{proof}

\begin{lemma} 
Suppose $C$ consists of three double  lines, not passing through one point  or $C$ consists of the union of a triple line with a double and single line, not all three passing through one point. Then there is no point with non-trivial local cohomology.
\end{lemma}

\begin{proof}
Note that all intersections are transversal. Hence the result follows directly from Lemma~\ref{lemNull}.
\end{proof}

We still need to determine $H^4_p(Y)$ for singularities of type $(A_k,m)$. Note that for $(A_k,4), k\geq 4$ we have local equations
\[ (t+s^2)^2(t^2+s^{k+1}) \]
which are not weighted homogeneous. For singularities of type $(A_2,k)$, for $k\in\{3,4\}$ we have local equations 
\[ t^2(ts+(t-s)^k) \]
which are not weighted homogeneous. In total we have six types of singularities for which we do not have a method to calculate $H^ 4_p(Y)$.

It remains to consider the cases $(A_k,2)$, for $1\leq k\leq7$,  $(A_2,3)$ , $(A_3,4)$ and the case that $Q$ is smooth at the intersection points with $\ell$, and $\ell$ is a bitangent or a quadruple tangent to $Q$.

\begin{lemma} Suppose we have a singularity of the form
 \[y^2=x^3+t^2(t+s^{2k})\]
 with $k\in \{1,2\}$. Then $H^4_p(Y)$ is two-dimensional.\end{lemma}
\begin{proof}
Setting weights $1,2k,2k,3k$, yields $2d-w=5k-1$. Clearly, the degree $2d_p-w_p$ part of $R(f_p)_{2d_p-w_p}$ is spanned by 
$s^{5k-1},xs^{3k-1},ts^{3k-1},xts^{k-1}$. At $(1:0:0:0)$ we have an $A_2$ singularity. The images of $s^{5k-1}$ and $xs^{3k-1}$ generate the local Milnor algebra, hence $H^4_p(Y)$ is 2-dimensional.
\end{proof}

\begin{lemma} Suppose we have an $(A_k,2)$ singularity then $H^4_p(Y)$ is non-zero if and only if $k\in \{ 3,6\}$. If $k\in \{3,6\}$ then $H^4_p(Y)=\C(-2)$.\end{lemma}
\begin{proof}
A local equation is of the form
\[y^2=x^3+s^2(t^2+s^{k+1})
\]
Setting weights  $6,3k+3,2k+6,3k+9$ shows that we can apply the three weight trick  if $3 \nmid 2k+6$. Hence if $k\not \equiv 0 \bmod 3$ then $H^4_p(Y)=0$. 

If $k=3$, we have that $R_{2d_p-w_p}$ is spanned by the images of $xs^2,xt,s^4,t^2$. The local Milnor algebra at $y=x=s=0$ is generated by $1,x$, hence $\tilde{R}_{2d_p-w_p}$ is spanned by $xs^2, s^4$ and $h^4_p(Y)=2$.

If $k=6$, we have that $R_{2d_p-w_p}$ is spanned by the images of $xs^4,s^6$. The local Milnor algebra at $y=x=s=0$ is generated by $1,x$, hence $\tilde{R}_{2d_p-w_p}=R_{2d_p-w_p}$ is spanned by $xs^3, s^6$ and $h^4_p(Y)=2$.
\end{proof}

\begin{lemma} Suppose we have an  $(A_2,3)$ or an $(A_3,4)$ singularity then $H^4_p(Y)=0$.
\end{lemma}
\begin{proof}
In the first case we have local equation $y^2=x^3+t^2(t^2+s^3)$.  Setting weights $3,2,4,6$, shows that we can apply the three weights trick to reduce to the singularity $y^2=x^3+s(s+t^3)$ with weights $3,1,2,3$. From Lemma~\ref{lemNull} it follows that this singularity has no local cohomology.

In the second case we have local equation $y^2=x^3+t^2(t^2+s^4)$. Setting weights $6,3,8,12$ shows that  we can apply the three weights trick. Hence there is no local cohomology.
\end{proof}

Finally, in the case of a triple line and a cubic curve we have the following singularity.

\begin{lemma} Suppose $(Y,p)$ is a singularity of type
\[y^2=x^3+t^3(t+s^3).\] 
Then $H^4_p(Y)$ is two-dimensional.
\end{lemma}

\begin{proof}
If we set weights of $s,t,x,y$ to be $1,3,4,6$, we obtain $2d-w=10$. The vector space $R(f_p)_{10}$ is spanned by $xs^6,xts^3, xt^2, t^2s^4 ts^7 s^{10}$. We have a singularity at $(1:0:0:0)$. The stabilizer at this point is trivial and the  Milnor algebra is generated (in global coordinates) by $s^{10},xs^6,ts^7,xts^3$, hence $\tilde{R}$ is generated by $t^2s^4,xt^2$.
\end{proof}

\section{Determining the Mordell-Weil rank} \label{secMWrank}
To determine the Mordell-Weil rank of an elliptic threefold we use the main result of \cite{ell3HK}:
Let $Y\subset \Ps(2,3,1,1,1)$ be a hypersurface given by
\[ y^2=x^3+Px+Q,\]
where $P$ and $Q$ are polynomials in $z_0,z_1,z_2$ of degree 4 and 6 respectively.

Let $\psi: Y\dashrightarrow \Ps^2$ be the projection from $(1:1:0:0:0)$ onto the plane $\{x=y=0\}$. The map $\psi$ is not defined at $(1:1:0:0:0)$. Let $X_0$ be the blow-up of $Y$ at $(1:1:0:0:0)$. This blow-up resolves the singularity of $\psi$ and endows $X_0$ with the structure of a Weierstrass fibration in the sense of Miranda. Miranda gave a description of which birational transformations one needs to apply in order to obtain an elliptic threefold $\pi: X\to S$.

The torsion part of $\MW(\pi)$ can be determined by specialization and we will come back to this later.
In \cite{ell3HK} we gave together with Klaus Hulek a procedure that for general $Y$ calculates $\rank \MW(\pi)$.  To determine the rank of $\MW(\pi)$ one can use that if $H^4(Y,\C)$ has a pure weight 4 Hodge structure then
\[ \rank \MW(\pi)  = \rank H^{2,2}(H^4(Y,\C)) \cap H^4(Y,\Z)-1 .\]
In general intersections of the type $H^{2,2}(H^4(Y,\C)) \cap H^4(Y,\Z)$ are hard to calculate. An exception is the case where $H^4(Y,\C)=H^{2,2}(Y,\C)$.
This is actually always the case in all our examples:

\begin{lemma}\label{lemPure} Suppose every non-isolated singularity of $Y$ is weighted homogeneous. Then $H^4(Y,\C)$ is pure of type $(2,2)$.
\end{lemma}

\begin{proof} Consider the exact sequence
\[\dots \to \oplus_{p\in \Sigma} H^ 4_p(Y) \to H^ 4(Y) \to H^4(Y^ *). \]
We start by proving that the mixed Hodge structure on $H^4(Y)$ is pure of weight 4.
Since $Y^ *$ is smooth, it follows that $H^4(Y^ *)$ has only Hodge weights $\geq 4$, whereas $H^4(Y)$ has only Hodge weights $\leq 4$, since $Y$ is proper (both statements can be found in \cite[Section 5.3]{PSbook}). Hence to prove the above claim, it suffices to prove that the Hodge structure on $H^4_p(Y)$ is of pure weight 4.

Suppose $p\in Y_{\sing}$ and suppose we have a weighted homogeneous singularity at $p$. Then by the results of Dimca \cite{DimBet} and of \cite{ell3HK}, it follows that $H^4_p(Y)$ has pure weight 4.
If $(Y,p)$ is not weighted homogeneous then this singularity is isolated. The procedures in the Singular library \verb!gmssing.lib! allow us to calculate the weight filtration on $H^4_p(Y)$. It turns out that for all singularities mentioned in Theorem~\ref{ThmSextic} the Hodge structure on $H^4_p(Y)$ is pure of  weight 4. From this it follows that $H^4(Y)$ is pure of weight 4.

In order to prove that $H^4(Y)$ is pure of  type $(2,2)$, consider  $f:\tilde{Y}\to Y$, a resolution of singularities of $Y$. Let $\ell\subset \Ps^2$ be a general line. Then  $Y_\ell:= f^{-1}(\psi^{-1}(\ell))$ is irreducible and is a rational elliptic surface. Moreover, since $\ell$ is ample, we have by Lefschetz' hyperplane theorem that  $H^2(\tilde{Y})\to H^2(Y_\ell)$ is injective. From the rationality of $Y_\ell$  it follows that $h^{2,0}(Y_\ell)=0$ and therefore $h^{2,0}(\tilde{Y})=0$. Using Poincar\'e duality one obtains $h^{3,1}(\tilde{Y})=h^{1,3}(\tilde{Y})=0$. In particular $H^{2,2}(\tilde{Y})=H^4(\tilde{Y})$.

We have an exact sequence $H^3(E) \to H^ 4(Y) \to H^4(\tilde{Y})$. Since $E$ is proper it turns out that there the graded piece of weight 4  in $H^ 3(E)$ is trivial.  Since $H^4(Y)$ is pure of  weight 4 this exact sequence implies  that $H^4(Y)$ injects in $H^4(\tilde{Y})$. The latter Hodge structure is pure  of type $(2,2)$, so the same holds for $H^4(Y)$.
\end{proof}

\begin{proposition}\label{prpSurj}
Suppose $(Y,p)$ is a semi-weighted homogeneous hypersurface singularity. Then $H^4(\Ps \setminus Y,\C) \to H^4_p(Y)$ is surjective.
\end{proposition}

\begin{proof}
Suppose first that $j=0$. Then there exist positive integers $d$ (divisible by $6$), $v_1$, $v_2$, $\alpha$, $\beta$, $\gamma$, $\delta$ such that $v_1\alpha+v_2\beta=v_1\gamma+v_2\delta=d$, and the $\gcd$ of  $d/6$, $v_1$ and $v_2$ equals $1$ and
$(Y,p)$ is locally given by
$y^2+x^ 3+s^{\alpha}t^\beta+s^{\gamma}t^{\delta}$ plus terms of the same or higher (weighted) degree. Moreover, since $C$ is a sextic we may assume that $\alpha+\beta$ and $\gamma+\delta$ are at most 6.

If both $v_1$ and $v_2$ are divisible by 2 then three of the weights are divisble by $2$ and we can apply the 3 weights trick and obtain that $H^4_p(Y)=0$. The same conclusion holds if both $v_1$ and $v_2$ are divisible by $3$.


For all other choices of $(v_1,v_2)$ we used the computer program Singular to calculate $2d-w$ and $d-w$. If $\C[x,y,s,t]_{2d-w}$ is spanned by elements of (usual) degree at most 4, then the map $H^4(U)\to H^4_p(Y)$ is surjective. The only triples $(d,v_1,v_2)$ not satisfying this criterion are $d=12,v_1=1,v_2=3$ and $d=12,v_1=1,v_2=4$. Since the singularity lies on a sextic it turns out that this corresponds to the singularities
\[ y^ 2=x^3+t^3(t+s^3) \mbox{ resp. } y^ 2=x^ 3+t^2(t+s^4) .\]
For both singularities we know  $H^4_p(Y)=0$. 

The case $j=1728$ can be treated similarly, but turns out to be easier. This finishes the proof.
\end{proof}

Summarizing, we have that  $\rank \MW(\pi)=h^4(Y)-1$, that $h^ 4(Y)-1$ equals the dimension of the cokernel of
\[ H^4(\Ps\setminus Y,\C) \to \oplus_{p\in \Sigma} H^4_p(Y)\]
and that if $\Sigma$ consists of one point at which $Y$ has a weighted homogeneous singularity then this cokernel is trivial.

To calculate in practice the cokernel we might use that this cokernel ifs of pure $(2,2)$-type, hence it suffices to calculate
\[ \coker \left(\Gr_F^2 H^4(U,\C) \to \Gr_F^2 H^ 4_p(Y)\right). \]
In the sequel, we will only  calculate the rank in the case that $(Y,p)$ is weighted homogeneous, hence for the rest of this section assume that $Y$ has only weighted homogenous singularities.
In the previous section we showed for each weighted homogeneous singularity that $H^4_p(Y)$ is pure of type $(2,2)$. Hence it suffices to calculate
\[ \coker \left(\Gr_F^2 H^4(U,\C) \to  H^ 4_p(Y)\right). \]
There is a natural map $\C[x,z,z_0,z_1,z_2]_{4}\to \Gr_F^2 H^4(U,\C)$ given by
\[ g\mapsto \frac{g}{f^2} \Omega. \]
Here $f$ is a defining equation for $Y$ and $\Omega$ is the ``standard'' $4$-form on $\Ps$ (cf. \cite[Section 5]{ell3HK}).
The Jacobian ideal lies in the kernel of this map (see e.g., \cite{DimBet}). 
Since $y$ is in the Jacobian ideal, we get a surjection $\C[x,z_0,z_1,z_2]_4 \to H^ 4(U,\C)$.

The map $H^4(U,\C) \to H^4_p(Y,\C)$ can be calculated as follows. In the previous section we provided generators $g_1,\dots g_k$ for $H^4_p(Y,\C)$.
Now the map $\C[z_0,z_1,z_2]_2x\oplus \C[z_0,z_1,z_2]_4 \to H^4_p(Y,\C)$ is given by
\[G \to \left(\frac{\partial G}{\partial g_1}(p),\dots,\frac{\partial G}{\partial g_k}(p)\right).\]
We can simplify the calculation of the Mordell-Weil rank further: the only interesting cases are $j(\pi)=0,1728$. In that case the fibration with section  has  an extra automorphism, namely \[\omega:(x,y,z_0,z_1,z_2)\to (\omega x,y,z_0,z_1,z_2)\] with $\omega^2=-\omega-1$ (if $j(\pi)=0$ or  \[i:(x,y,z_0,z_1,z_2)\to (-x,iy,z_0,z_1,z_2) \mbox{ if } j(\pi)=1728.\] Let $\sigma$ either be $\omega$ or $i$.
The action of $\sigma$ gives $\MW(\pi)$ the structure of a $\Z[\sigma]$-module. In particular the $\Z$-rank of $\MW(\pi)$ is twice the $\Z[\sigma]$-rank of $\MW(\pi)$.
If we fix a basis  $P_1,\dots,P_r$ for $\MW(\pi)/\MW(\pi)_{\tor}$ as $\Z[\sigma]$-module, then $P_1,\sigma P_1,\dots, P_r,\sigma P_r$ is a basis for $\MW(\pi)/\MW(\pi)_{\tor}$ as $\Z$-module.

Then $\sigma$ acts on  $P_i,\sigma P_i$ as
\[ \left(\begin{matrix} 0 & -1 \\1&-1 \end{matrix}\right) \mbox{resp. } \left(  \begin{matrix} 0 & -1 \\1&0 \end{matrix} \right) \]
This implies that on $\MW(\pi)\otimes_\Z \overline{\Q}$ the only eigenvalues of $\sigma$ are $\omega,\omega^2$ resp. $i,-i$, and the corresponding two eigenspaces have the same dimension.

The automorphism $\sigma$ induces actions on $H^4(Y,\C)_{\prim}$, $H^4_p(Y,\C)$ and the graded piece $\Gr_F^k H^4(U,\C)$. Recall that we are interested in the calculation of the cokernel of
\[ F^3 H^4(U,\C) \to \oplus_{p\in \cP} H^4_p(Y). \]
The cokernel is a direct sum of the two  eigenspaces and both eigenspaces have the same dimension. Hence it suffices the calulate the dimension of the $\omega^2$ (resp. $i$) eigenspace of the cokernel.

Since $\sigma(\Omega)=\omega\Omega$ if $j(\pi)=0$ (resp. $-i \Omega$ if $j(\pi)=1728$) and $F^3H^4(U,\C)$ is a quotient of \[(x\C[z_0,z_1,z_2]_2 \oplus \C[z_0,z_1,z_2]_4)\cdot \frac{1}{f^2} \Omega,\] it follows that the $\omega^2$-eigenspace, respectively,  the $i$ eigenspace is the co-kernel of
\[ x\C[z_0,z_1,z_2]_2 \cdot \frac{1}{f^2} \Omega \to \oplus H^4_p(Y,\C)^{\sigma-\omega^2, \sigma+i}.\]

At the level of the local cohomology the same phenomena happens i.e., $\sigma$ acts on monomials of the form $xh(t,s)$ as multiplication by $\omega^2$ resp. $i$, and on monomials of the form $h(t,s)$ it acts as $\omega$, resp. $-i$.

\begin{remark} It should be remarked that on $H^4_p(Y)$ the two eigenspaces have the same dimension. However, on $F^3H^4(U,\C)$ the two eigenspaces have different dimensions, namely 6 and 15. For computational reasons we choose to work with the 6-dimensional space.
\end{remark}

\section{Classification I: $j(\pi)=1728$} \label{secj1728}
 This case is rather easy.

\begin{lemma}
Suppose $j(\pi)=1728$. Then $\MW(\pi)_{\tor}\not \cong \Z/2\Z $ if and only if $C$ is a double conic. If $C$ is a double conic then $\MW(\pi)_{\tor}=(\Z/2\Z)^ 2$.
\end{lemma}

\begin{proof} From Lemma~\ref{lem2tor} it follows that $\Z/2\Z$ is a subgroup of $\MW(\pi)$.
Suppose that $\# \MW(\pi)_{\tor}>2$, then
for a general line $\ell$ the intersection $C\cap\ell$ consists of two points with multiplicity 2 by Proposition~\ref{prpSur1728}. Hence $C$ is a double conic. Conversely, if $C$ is a double conic, then $Y$ is given by $y^2=x^3+f^2x$. This threefold has $x=\pm f,y=0$ and $x=0,y=0$ as sections of order 2. Hence $\MW(\pi)_{\tor}\supset(\Z/2\Z)^2$. From Corollary~\ref{corMWj1728} it follows that then $\MW(\pi)_{\tor}=(\Z/2\Z)^2$.
\end{proof}

\begin{theorem} Suppose $j(\pi)=1728$ and that 
$C_{\red}$ is not the union of lines through one point. Then $\MW(\pi)$ is infinite if and only if $C$ is a quartic with two $A_3$ singularities.

Moreover, we have 
\begin{itemize}
\item $\MW(\pi)\cong \Z/4\Z$ if and only if $C$ is a double conic,
\item $\MW(\pi)\cong \Z/2\Z\times \Z^2$ if and only if $C$ is a quartic with two $A_3$ singularities.
\item $\MW(\pi)\cong \Z/2\Z$ otherwise.
\end{itemize}
\end{theorem}

\begin{proof}
Suppose $C$ is a quartic with two $ A_3$ singularities. A smooth degree 4 curve has Euler characteristic $-4$. Since the Milnor number of an $A_3$ singularity is 3, we obtain that $C$ has Euler characteristic $-4+6=2$, hence $h^1(C)=0$, because  $h^0(C)=h^2(C)=1$. This implies that $C$ is a rational curve. Hence without loss of generality we may assume that $C$ is given by $z_0^4-z_1^2z_2^2$. It remains to show that
\[y^2=x^3-(z_0^4-z_1^2z_2^2)x\]
has Mordell-Weil rank 2. Since $h^4_p(Y)=2$ and $\Sigma$ consists of two points, we have  $\rank \MW(\pi)\leq 4$. From the surjectivity of $H^4(U)\to H^4_p(Y)$  (Proposition~\ref{prpSurj}) it follows that the cokernel $H^4(U)\to H^4_\Sigma(Y)$ has dimension at most 3, and,  since this dimension is even, it follows that $\rank \MW(\pi)\in \{0,2\}$. Note that $x=z_0^2, y=z_0z_1z_2$ is a point of infinite order. Hence $\rank \MW(\pi)=2$.

Conversely,  we have that $H^4_p(Y,\C)$ is non-zero if and only if $(C,q)$ is an isolated singularity if type $A_3, A_7$ or $D_7$. Since $H^4(U,\C)\to H^4_p(Y,\C)$ is surjective for each such singularity, we need to have at least two singularities for positive rank.
This means that $C$ is a quartic with 2 $A_3$ singularities.

To finish the proof, note that from Corollary~\ref{corMWj1728}  implies that if $\MW(\pi)$ is finite then it is isomorphic to $\Z/2\Z$ or $(\Z/2\Z)^2$. From the previous lemma  it follows that the latter only occurs if  $C$ is a double conic.
\end{proof}

\section{Partial classification: case $j(\pi)=0$ and $C$ is non-reduced} \label{secNonRed}
Suppose $C$ is a non-reduced sextic. 
Consider first the case that $C$ is a reduced quartic with a double line. 
In this case we cannot calculate $H^4_p(Y)$ for six types of  the singularities that occur in this case. For this reason we give  we give a few examples with positive rank.

\begin{example}
Suppose $C$ is the union of a double line $\ell$ and a quartic $Q$. Then $\MW(\pi)$ has rank 2 if one of the following occurs
\begin{itemize}
 \item $C$ has an $E_6$ singularity, and $\ell$ intersects $Q$ with multiplicity 4 in a smooth point or
\item $C$ has two $A_3$ singularities along $\ell$.

\end{itemize}
\end{example}

\begin{proof}
In the first case we may assume that, after a change of coordinates if necessary, $Y$ is given by  $y^2=x^3+z_0^2(z_1^4+z_0z_2^3)$. Since $H^4_\Sigma(Y)$ is four-dimensional, $H^4(U)\to H^4_\Sigma(Y)$ is not the zero map, and the cokernel has even dimension, we have that $\rank \MW(\pi)\in\{0,2\}$. Now $x=z_0z_2$ and $y=z_0z_1^2$ is  a point of infinite order, showing that $\rank \MW(\pi)=2$.

In the case we may assume that, after a change of coordinates if necessary, $Y$ is given by $y^2=x^ 3+z_0^2(z_0^ 4+z_1^2z_2^2)$.
Since  $C$ has two $A_3$ singularities it follows that $H^4_\Sigma(Y)$ is four-dimensional. By the same reasoning as above we have that $\rank \MW(\pi) \in \{0,2\}$. The point $x=z_0z_1z_2, y=z_0z_1^2$ has clearly inifinite order, hence the rank equals 2.\end{proof}

From the results in Section~\ref{secCalc} it follows that there are non-reduced sextics, not being a double line with a quartic, that might yield elliptic threefolds with positive rank. In all these cases it turns out that the rank equals 2.

\begin{theorem} Suppose $C$ is one of the following
\begin{itemize}
\item $C$ is a triple line $\ell$ together with cuspidal cubic $K$, and $\ell$ is a flex line at a smooth point of $K$,
\item $C$ is a conic together with two double lines $\ell_1,\ell_2$, such that the $\ell_{i,\red}$ are tangent to $C$ or
\item $C$ is a conic together $C_1$ with a double conic $C_2$, and  $C_1$ and $C_{2,\red}$ intersect in precisely two points with multiplicity 2.
\end{itemize}
Then $\MW(\pi)=\Z^2$.
\end{theorem}

\begin{proof} Using a specialization argument it follows that in all these cases the torsion part is trivial.
In all cases $\Sigma$ consists of two points, and both points have $h^4_p(Y)=2$. The map  $H^4(\Ps\setminus Y)\to H^4_{p_1}(Y) \oplus  H^4_{p_2}(Y) $ is not the zero map by Proposition~\ref{prpSurj}, hence the cokernel has dimension at most 3 and therefore $\rank \MW(\pi)\leq 3$. Since the rank is even, one has  $\rank \MW(\pi)\in \{0,2\}$. In order to prove the results it suffices to give a non-trivial section.

In the first case, without loss of generality we may assume that   $Y$ is given by
\[ y^2=x^3+z_0^3(z_0^2z_1-z_2^3).\]
Then the section $x=z_0z_2$, $y=z_0^2z_1$ is non-torsion. 

In the second case, without loss of generality we may assume that   $Y$ is given by
\[ y^2=x^3+(z_0^2+z_1z_2)z_1^2z_2^2.\]
Then the section $x=-z_1z_2$, $y=z_0z_1z_2$ is non-torsion. 

In the third case, without loss of generality we may assume that  without loss of generality  $Y$ is given by
\[ y^2=x^3+(z_0^2+z_1z_2)(\alpha z_0^2+z_1z_2)^2,\]
with $\alpha \in \C$. The section $x=(\alpha z_0^2+z_1z_2)$, $y=(\sqrt{1-\alpha}) z_0^2(\alpha z_0^2+z_1z_2)$ is non-torsion. 
\end{proof}

\section{Case $j(\pi)=0$ and $C$ is a cuspidal curve} \label{secCusps}

Suppose $C$ is a sextic with only cusps.
It is well-known that $C$ can have at most 9 cusps. Moreover, at most 3 of such cusps can lie on a line and at most 6 of them on a conic.

We need the following lemma:
 \begin{lemma}Let $\{p_1,\dots,p_m \}$, $m\leq 9$ be a set of distinct points in $\Ps^2$, with no four points collinear and no seven points lying on the same conic.
 Let $K$ be the cokernel of  the  evaluation map at $p_1,\dots,p_m$:
\[\phi: \C[z_0,z_1,z_2]_2\to \C^m.\]
Then $\dim K=m-6$ for $m\geq 7$, and $\dim K=0$ for $m\leq 5$.  For $m=6$ we have $\dim K=1$ if all the points lie on a conic, $\dim K=0$ otherwise.
 \end{lemma}
\begin{proof}
If $m\geq 7$ then the $m$ points do not lie on a conic, hence the kernel of $\phi$ is trivial and the cokernel has dimension $6-m$.

If $m=6$ and the points do not lie on a conic then the kernel of $\phi$ is again trivial and $\dim K=0$.

If $m=6$ and the points do lie on a conic then the kernel of $\phi$ is one-dimensional and so is the cokernel.

If $m<6$ then $K$ is non-trivial only if the elements in the kernel have a common component. Such a component is necessarily a line and $m\geq 3$. 
A straightforward calculation shows that if $3\leq m \leq 5$ and precisely three of the $m$ points are collinear then the kernel of $\varphi$ has dimension $6-m$, so $\dim K=0$.
\end{proof}

Let $Y$ be an elliptic threefold of the form  $y^2=x^3+f(z_0,z_1,z_2)$ where $f=0$ is a reduced sextic with only cusps as singularities. For each cusp $p_i$ of $f=0$ fix a direction $\ell_i$ such that $C$ intersects $\ell_i$ with multiplicity 3 at $p_i$.

In Lemma~\ref{lemj0Ak}
 we studied the singularity $y^2=x^3+t^3+s^2$. It turns out that $H^4_p(Y)$ is generated by the class of $x$ and $t$.

This implies that we can determine the cokernel of the map $H^4(U,\C) \to H^4_\Sigma(Y)$ as follows:
\[ x\C[z_0,z_1,z_2]_2 \oplus \C[z_0,z_1,z_2]_4 \to (\C^2)^m \]
 where $(xf_2+f_4)$ is mapped to $(f_2(p_i), \frac{\partial}{\partial\ell_i}f_4)$.
 To simplify matters we can decompose the cokernel into eigenspaces for the complex multiplication. One eigenspace is the cokernel of
\[  \C[z_0,z_1,z_2]_4 \to \C^m, f_4\to\frac{\partial}{\partial\ell_i}f_4,\] where the other is the cokernel of
\[ x\C[z_0,z_1,z_2]_2  \to \C^m, xf_2 \mapsto f_2(p_i) \]
By the above lemma, this map has one dimensional cokernel if $m=6$ and the cusps lie on a conic or $m=7$, a two-dimensional cokernel if $m=8$ and a three-dimensional cokernel of $m=9$. The latter case is well known, it means that the curve $C$ is the dual of a smooth cubic.

Since both eigenspaces have equal dimension we obtain the following result.
\begin{theorem}
Let $f=0$ be a reduced sextic, with only cusps as singularities. Suppose the cusps are at $p_1,\dots,p_m$. Then the elliptic threefold
\[ y^2=x^3+f\]
has the following Mordell-Weil group:
\begin{itemize}
\item If $m\leq 5$ or $m=6$ and the $p_i$ do not lie on a conic then $\MW(\pi)=0$.
\item If $m=6$ and the  $p_i$ lie on a conic then $\MW(\pi)=\Z^2$.
\item If $m\geq 7$ then $\MW(\pi)=\Z^{2(m-6)}$.
\end{itemize}
\end{theorem}
In particular, this shows the existence of the Mordell-Weil groups $\Z^{2r}$ for $r=0,1,2,3$.

\begin{remark} \label{TorRem} Suppose $C$ is a sextic with $9$ cusps. Then $C$ is the dual curve of a smooth cubic. Hence there is a one-dimensional family of sextics with $9$ cusps, and hence a one-dimensional family of elliptic threefolds with Mordell-Weil rank 6. Since the Mordell-Weil rank is six and $H^4(Y)$ is pure of type $2,2$ it follows that $H^4(Y,\Q)=\Q(-2)^7$. All other cohomology groups, except for $H^ 3$, can be calculated using the Lefschetz hyperplane theorem, i.e., $H^{2i}(Y,\Q)=\Q(-i)$ for $i=0,1,3$ and $H^i(Y,\Q)=0$ for $i\not \in \{0,2,3,4,6\}$.

As explained in \cite[Section 3]{ellrat} it follows that $H^3(Y,\Q)=\Q(-1)^{12}$. All cohomology groups have Hodge structures of Tate type, and there is no variation of Hodge structures possibile. In particular, a Torelli type result as obatined by Grooten-Steenbrink \cite[Section 6]{GrSte} in a similar setting is not possible in our case.
\end{remark}

\section{Possible Mordell-Weil groups}\label{SecPos}
In the previous section we have seen the existence of the groups $\Z^{2r}$ for $r=0,1,2,3$. In order to prove Theorem~\ref{MWGrpThm} we have to show the existence of the groups $\Z/3\Z, (\Z/2\Z)^2$.

\begin{remark} We have that  $\Z/3\Z\subset \MW(\pi)$ if and only if $Y$ has an equation of the form
\[y^2=x^3+f^2\]
where  $f=0$ is a cubic.
We showed Lemma~\ref{lem3tors} that then $\MW(\pi)=\Z/3\Z $ unless $f=0$ is the union of three lines, and since we have excluded the cone construction case, $f=0$ is the union of three lines $\ell_1,\ell_2,\ell_3$ without a common intersection point. That means that $\Sigma$ consists of three points $\{p_1,p_2,p_3\}$ and at each point we have a local equation $y^2=x^3+(ts)^2$.  As explained in Lemma~\ref{lemNull}, we have that  $H^4_{p_i}(Y)=0$, whence $\MW(\pi)=\Z/3\Z$ in this case, and $\Z/3\Z\times \Z^2$ is not possible.
\end{remark}

\begin{remark} Suppose we have that  $\MW(\pi)=\Z^8$. We showed before that than $C$ is a reduced sextic, and is a union of six lines, not through one point. That means that for each $p\in \Sigma$ we have a local equation of the form
 $y^2=x^3+t^m+s^m$ with $2\leq m \leq 5$. For each such singularity we have $H^4_p(Y)=0$,  so if $C$ is the union of lines then $\MW(\pi)$ is finite. This shows that $\Z^8$ is not possible.
\end{remark}

Summarizing we get:
\begin{theorem} Let $y^2=x^3+f$ be an elliptic threefold, $f=0$ is not the union of lines through one point.
\begin{itemize}
\item $\MW(\pi)\cong (\Z/2\Z)^2$ if and only if $f=0$ is triple conic.
\item $\MW(\pi)\cong (\Z/3\Z)$ if and only if $f=0$ is double cubic.
\item Otherwise $MW(\pi)$ is one of $0, \Z^2,\Z^4,\Z^6$, and all these cases occur.
\end{itemize}
\end{theorem}


\end{document}